\documentclass[12pt]{amsart}

\usepackage{amssymb,amsmath,amsthm, url}
\usepackage[hypertex]{hyperref}
\usepackage{enumitem}
\usepackage[all]{xy}

\setenumerate{label=(\alph*), font=\normalfont}

\newtheorem{thm}[equation]{Theorem}
 \newtheorem{prop}[equation]{Proposition}
 \newtheorem{lem}[equation]{Lemma}
 \newtheorem{cor}[equation]{Corollary}

 \theoremstyle{definition}
 \newtheorem{example}[equation]{Example}
 \newtheorem{defn}[equation]{Definition}
 \newtheorem{remark}[equation]{Remark}

\numberwithin{equation}{section}

\newcommand{\bbZ}{{\mathbb{Z}}}
\newcommand{\bbR}{{\mathbb{R}}}
\newcommand{\bbP}{{\mathbb{P}}}

\newcommand{\bbQ}{{\mathbb{Q}}}
\newcommand{\bbH}{{\mathbb{H}}}

\newcommand{\Mat}{\operatorname{M}}
\newcommand{\GL}{\mathrm{GL}}

\newcommand{\Rep}{\mathbf{Rep}}
\newcommand{\V}{\mathbf{V}}
\newcommand{\rank}{\operatorname{rank}}
\newcommand{\ed}{\operatorname{ed}}

\newcommand{\trdeg}{\operatorname{trdeg}}
\newcommand{\Gal}{\operatorname{Gal}}

\newcommand{\Sub}{\mathbf{Sub}}

\newcommand{\Env}{\operatorname{Env}}
 \newcommand{\ind}{\operatorname{ind}}
 \newcommand{\tr}{\operatorname{tr}}
 \newcommand{\M}{\operatorname{M}}
\newcommand{\SB}{\operatorname{S\!B}\nolimits}

\newcommand{\Br}{\mathop{\mathrm{Br}}}

\newcommand{\cd}{\mathop{\mathrm{cd}}\nolimits}

\newcommand{\Char}{\mathop{\mathrm{char}}\nolimits}

\newcommand{\Z}{\mathbb{Z}}

\newcommand{\Spec}{\operatorname{Spec}}
\newcommand{\End}{\operatorname{End}}

\newcommand{\Prod}{\operatornamewithlimits{\textstyle\prod}}

\newcommand{\Fields}{\mathbf{Fields}}

\newcommand{\Sets}{\mathbf{Sets}}

\renewcommand{\phi}{\varphi}

\newcommand{\WR}{R}

\author
{Nikita A. Karpenko}

\address
{Mathematical \& Statistical Sciences \\
University of Alberta \\
Edmonton
\\
CANADA}

\email
{karpenko {\it at} ualberta.ca, {\it web page}: www.ualberta.ca/\~{ }karpenko}

\thanks
{The first author acknowledges partial support of the French Agence Nationale
de la Recherche (ANR) under reference ANR-12-BL01-0005;
his work has been also partially supported by a start-up grant of
the University of Alberta and a Discovery Grant from
the National Science and Engineering Board of Canada.}

\author{Zinovy Reichstein}

\address
{Department of Mathematics \\
University of British Columbia \\
Vancouver
\\
CANADA}

\email
{reichst {\it at} math.ubc.ca, {\it web page}: www.math.ubc.ca/\~{ }reichst}

\thanks
{The second author has been partially supported by a Discovery Grant from
the National Science and Engineering Board of Canada}

\date
{June 17, 2014}

\begin{document}

\title
[An invariant for representations of finite groups]
{A numerical invariant\\
 for linear representations\\
  of finite groups}

\keywords
{Representations of finite groups, characters,
Schur index, central simple algebras, essential dimension,
Severi-Brauer varieties, Weil transfer,
Chow groups and motives, canonical dimension and incompressibility}
\subjclass[2010]{14C25, 16K50, 20C05}

\begin{abstract}
We study the notion of essential dimension for a linear
representation of a finite group.  In characteristic zero
we relate it to the canonical dimension
of certain products of Weil transfers of generalized
Severi-Brauer varieties. We then proceed to compute the canonical
dimension of a broad class of varieties of this type, extending
earlier results of the first author. As a consequence,
we prove analogues of classical theorems of R.~Brauer and O.~Schilling
about the Schur index, where the Schur index of a representation
is replaced by its essential dimension.
In the last section we show that essential dimension
of representations can behave in rather unexpected ways in
the modular setting.
\end{abstract}

\maketitle

\section{Introduction}
Let $K/k$ be a field extension,
$G$ be a finite group of exponent $e$, and
$\rho \colon G \to \GL_n(K)$ be a non-modular representation
of $G$ whose character takes values in $k$. (Here ``non-modular''
means that $\Char(k)$ does not divide $|G|$.)
A theorem of Brauer says that if $k$ contains a primitive
$e$th root of unity $\zeta_e$ then $\rho$ is defined over $k$, i.e., $\rho$
is $K$-equivalent to a representation
$\rho' \colon G \to \operatorname{GL}_n(k)$; see,
e.g.~\cite[\S12.3]{serre-representations}.
If $\zeta_e \not \in k$, we would like to know ``how far" $\rho$ is from
being defined over $k$.  In the case, where $\rho$ is absolutely
irreducible, a classical answer to this question is given by the Schur
index of $\rho$, which is the smallest degree of a finite
field extension $l/k$ such that $\rho$ is defined over $l$.
Some background material on the Schur index and further references
can be found in Section~\ref{sect.prel}.

In this paper we introduce and study another numerical invariant,
the essential dimension $\ed(\rho)$, which measures
``how far" $\rho$ is from being defined over $k$
in a different way. Here $\rho$ is not assumed to be irreducible;
for the definition of $\ed(\rho)$, see Section~\ref{sect.def-ed}.
In Section~\ref{sect.functor} we show that the maximal value of
$\ed(\rho)$, as $\rho$ ranges over representations with
a fixed character $\chi \colon G \to k$, which we denote by $\ed(\chi)$,
can be expressed as the canonical dimension of a certain product of
Weil transfers of generalized Severi-Brauer varieties.
We use this to show that
$\ed(\rho) \leq  |G|^2/4$
for any
$n$, $k$, and $K/k$ in Section~\ref{sect.upper-bounds}
and to prove a variant of a classical theorem of Brauer
in Section~\ref{sect.brauer-ed}.  In Section~\ref{sect.incompressible}
we compute the canonical dimension of a broad class of Weil
transfers of generalized Severi-Brauer varieties, extending earlier
results of the first author from~\cite{qweil}
and~\cite{upper}. This leads to a formula for the essential
$p$-dimension of an irreducible character in terms of its
decomposition into absolutely irreducible components;
see Corollary~\ref{cor.main2}.
As an application we prove a variant of a classical theorem
of Schilling in Section~\ref{sect.schilling-ed}. In the last
section we show that $\ed(\rho)$ can be unexpectedly
large in the non-modular setting.

\section{Notation and representation-theoretic preliminaries}
\label{sect.prel}

Throughout this paper $G$ will denote a finite group of exponent $e$,
$k$ a field, $\overline{k}$ an algebraic closure of $k$,
$K$ and $F$ field extensions of $k$, $\zeta_d$ a primitive $d$th
root of unity, $\rho$ a finite-dimensional
representation of $G$, and  $\chi$ a character of $G$.
In this section we will assume that $\Char(k)$ does not divide $G$.

\subsection{Characters and character values}
\label{sect.char}

A function $\chi \colon G \to k$ is
said to be {\em a character} of $G$, if $\chi$ is the character of
some representation $\rho \colon G \to \GL_n(K)$ for some
field extension $K/k$.

If $\chi \colon G \to \overline{k} $ is a character, and $F/k$ is a field,
we set
\[ F(\chi) := F(\chi (g)) \, | \, g \in G )  \subset F(\zeta_e). \]
Since $F(\zeta_e)$ is an abelian extension of $F$, so is
$F(\chi)$. Moreover, $F(\chi)$ is stable under automorphisms
$F(\zeta_e)/F$, i.e., independent of the choice of
the $e$th root of unity $\zeta_e$ in $\overline{F}$.

\begin{lem} \label{lem.abelian}
(a) Let $\chi, \chi' \colon G \to \overline{k}$ be characters and
$F/k$ be a field extension. Then

\smallskip
(a) every automorphism $h \in \Gal(F(\chi)/F)$
leaves $k(\chi)$ invariant.

\smallskip
(b) If $\chi$ and $\chi'$
are conjugate over $F$ then they are conjugate over $k$.

\smallskip
(c) Suppose $k$ is algebraically closed in $F$. Then the converse
to part (b) also holds. That is, if $\chi, \chi'$ are conjugate over $k$ then
they are conjugate over $F$.
\end{lem}

\begin{proof}
(a) It is enough to show that $h(\chi(g)) \subset F(\chi)$
for every $g \in G$. 
Since the sequence of Galois groups
\[ 1 \to \Gal(F(\zeta_e)/F(\chi)) \to \Gal(F(\zeta_e)/F) \to
\Gal(F(\chi)/F) \to 1 \]
is exact, $h$ can be lifted to an element of
$\Gal(F(\zeta_e)/F)$. By abuse of notation,
we will continue to denote it by $h$.
The eigenvalues of $\rho(g)$ are of the form
$\zeta_e^{i_1}, \dots, \zeta_e^{i_n}$ for some integers
$i_1, \dots, i_n$. The automorphism $h$ sends $\zeta_e$ to
another primitive $e$th root of unity $\zeta_e^j$ for some
integer $j$. Then
\[ h(\chi(g)) = h(\zeta_e^{i_1} + \dots + \zeta_e^{i_n}) =
\zeta_e^{ji_1} + \dots + \zeta_e^{ji_n} = \chi(g^j) \in F(\chi) \, , \]
as desired.

\smallskip
(b) is an immediate consequence of (a).

\smallskip
(c) If $k$ is algebraically closed in $F$, then
the homomorphism
$$
\Gal(F(\chi_1)/F) \to \Gal(k(\chi_1)/k)
$$
given by $\sigma \mapsto \sigma_{\, | k(\chi)}$ is surjective;
see~\cite[Theorem VI.1.12]{lang}.
\end{proof}

\subsection
{The envelope of a representation}
\label{sect.env}

If $\rho \colon G \to \GL_n(F)$ is a representation over some field $F/k$,
we define the $k$-{\em envelope} $\Env_k(\rho)$ as the $k$-linear span
of $\rho(G)$.

%
%

\begin{lem} \label{lem.multiple}
$\Env_k(s \cdot \rho)$ is $k$-isomorphic $\Env_k(\rho)$
for any integer $s \geqslant 1$.
\end{lem}

\begin{proof}
The diagonal embedding
$\M_n(F) \hookrightarrow \M_n(F) \times \dots \times \M_n(F)$
($s$ times) induces an isomorphism between $\Env_k(\rho)$
and $\Env_k(s \cdot \rho)$.
\end{proof}

\begin{lem} \label{lem.env}
Assume the character $\chi$ of $\rho \colon G \to \GL_n(F)$
is $k$-valued.  Then the natural homomorphism
$\Env_k(\rho) \otimes_k F \to \Env_F(\rho)$
is an isomorphism of $F$-algebras.
\end{lem}

\begin{proof}
It suffices to show that if $\rho(g_1), \dots, \rho(g_r)$
are linearly dependent over $F$ for some elements $g_1, \dots, g_r \in G$, then
they are linearly dependent over $k$. Indeed, suppose
\begin{equation*}
a_1 \rho(g_1) + \dots + a_r \rho(g_r) = 0
\end{equation*}
in $\M_n(F)$ for some
$a_1, \dots, a_r \in F$, such that $a_i \ne 0$ for some $i$.
Then
$$
\tr  ((a_1 \rho(g_1) + \dots + a_r \rho(g_r)) \cdot \rho(g)) = 0
$$
for every $g \in G$, which simplifies to
\begin{equation*}
a_1 \chi(g_1 g) + \dots + a_r \chi(g_r g) = 0.
\end{equation*}
The homogeneous linear system
\begin{equation*}
x_1 \chi(g_1 g) + \dots + x_r \chi(g_r g) = 0
\end{equation*}
in variables $x_1,\dots,x_r$ has coefficients in $k$
and a non-trivial solution in $F$.
Hence, it has a non-trivial solution $b_1,\dots,b_r$ in $k$,
and we get that
\begin{equation*}
\tr  ((b_1 \rho(g_1) + \dots + b_r \rho(g_r)) \cdot \rho(g)) = 0
\end{equation*}
for every $g \in G$.

Note that $\Env_k(\rho)$ is, by definition, a homomorphic image of
the group ring $k[G]$. Hence,  $\Env_k(\rho)$ is semisimple and
consequently, the trace form in $\Env_k(\rho)$ is non-degenerate.
It follows that the elements $\rho(g_1), \dots, \rho(g_r)$ are linearly dependent
over $k$, as desired.
\end{proof}

\subsection{The Schur index}
\label{sect.schur}

Suppose $K/k$ is a field extension, and
$\rho_1 \colon G \to \GL_n(K)$ is an absolutely
irreducible representation with character $\chi_1 \colon G \to K$.
By taking $F = \overline{K}$ in Lemma~\ref{lem.env},
one easily deduces that $\Env_{k(\chi_1)}(\rho_1)$ is
a central simple algebra of degree $n$ over $k(\chi_1)$.
The index of this algebra is called {\em the Schur index}
of $\rho_1$.  We will denote it by $m_{k}(\rho_1)$.

In the sequel we will need the following properties of the Schur index.

\begin{thm} \label{thm.schur} Let $K$ be a field, $G$ be a finite group
such that $\Char(K)$ does not divide $|G|$, and
$\rho \colon G \to \GL_n(K)$ be an irreducible representation.
Denote the character of $\rho$ by $\chi$.

\smallskip
(a) Over the algebraic closure $\overline{K}$, $\rho$ decomposes as
\begin{equation} \label{e.decomp}
\rho_{\overline{K}} \simeq m (\rho_1 \oplus \dots \oplus \rho_r),
\end{equation}
where $\rho_1, \dots, \rho_r$ are pairwise non-isomorphic irreducible
representations of $G$ defined over $\overline{K}$, and $m$ is their
common Schur index $m_K(\rho_1) = \dots = m_K(\rho_r)$.

\smallskip
(b) Let $\chi_i \colon G \to \overline{K}$ be the character
of $\rho_i$ for $i = 1, \dots, r$. Then
$K(\chi_1) = \dots = K(\chi_r)$ is an abelian extension
of $K$ of degree $r$.  Moreover, $\Gal(K(\chi_1)/K)$
transitively permutes $\chi_1, \dots, \chi_r$.

\smallskip
(c) Conversely, every irreducible representation
$\rho_1 \colon G \to \GL_1(\overline{K})$ occurs as an irreducible
component of a unique $K$-irreducible representation
$\rho \colon G \to \GL_n(K)$, as in~\eqref{e.decomp}.

\smallskip
(d) The center $Z$ of $\Env_K(\rho)$ is $K$-isomorphic to
$K(\chi_1) = K(\chi_2) = \dots = K(\chi_r)$.
$\Env_K(\rho)$ is a central simple algebra over $Z$ of index $m$.

\smallskip
(e) The multiplicity of $\rho_1$ in any representation
of $G$ defined over $K$ is a multiple of $m_K(\rho_1)$.
Consequently, $m_K(\rho_1)$ divides $m_k(\rho_1)$ for any
field extension $K/k$.

\smallskip
(f) $m$ divides $\dim(\rho_1) = \dots = \dim(\rho_r)$.
\end{thm}

\begin{proof}
See~\cite[Theorem 74.5]{curtis-reiner2} for parts (a)-(d),
and~\cite[Corollary 74.8]{curtis-reiner1} for parts (e) and (f).
\end{proof}

As a consequence of Theorem~\ref{thm.schur}, we obtain the following

\begin{cor} \label{cor.lift} Let $K/k$ be a field extension,
$\rho \colon G \to \GL_n(K)$ be a representation,
whose character takes values in $k$, and
 \[ \rho = d_1 \rho_1 \oplus \dots \oplus d_r \rho_r \]
be the irreducible decomposition of $\rho$ over the algebraic closure
$\overline{K}$.  Then $\rho$ can be realized over $k$
(i.e., $\rho$ is $K$-equivalent to
a representation $\rho' \colon G \to \GL_n(k)$) if and only if
the Schur index $m_k(\rho_i)$ divides $d_i$ for every $i = 1, \dots, r$.
\qed
\end{cor}

%

\section{Preliminaries on essential and canonical dimension}
\label{sect.ed}

\subsection{Essential dimension}
Let $ \mathcal{F}:\Fields_k\to\Sets $ be a covariant functor,
where $\Fields_k$ is the category of field
extensions of $k$ and $\Sets$ is the category of sets.
We think of the functor $\mathcal{F}$ as specifying the type
of algebraic objects under consideration, $\mathcal{F}(K)$
as the set of algebraic objects of this type defined over $K$,
and the morphism
$\mathcal{F}(i) \colon \mathcal{F}(K) \to \mathcal{F}(L)$
associated to a field extension
\begin{equation} \label{e.inclusion}
k \subset K \stackrel{i}{\hookrightarrow} L
\end{equation}
as ``base change".  For notational simplicity,
we will write $\alpha_L \in \mathcal{F}(L)$
instead of $\mathcal{F}(i)(\alpha)$,

Given a field extension $L/K$, as in~\eqref{e.inclusion},
an object $\alpha \in \mathcal{F}(L)$ is said to {\em descend} to $K$ if
it lies in the image of $\mathcal{F}(i)$.
The {\em essential dimension} $\ed(\alpha)$ is defined
as the minimal transcendence degree of $K/k$,
where $\alpha$ descends to $K$.  The essential
dimension $\ed(\mathcal{F})$ of the functor
$\mathcal{F}$ is the supremum of
$\ed(\alpha)$ taken over all $\alpha \in \mathcal{F}(K)$
and all $K$.

Usually $\ed(\alpha) < \infty$ for every $\alpha \in \mathcal{F}(K)$ and
every $K/k$;  see~\cite[Remark 2.7]{BRV2}.
On the other hand, $\ed(\mathcal{F}) = \infty$ in many
cases of interest; for example, see Proposition~\ref{prop.modular}
below.

The essential dimension
$\ed_p(\alpha)$ of $\alpha$ at a prime integer $p$ is defined
as the minimal value of $\ed(\alpha_{L'})$, as $L'$ ranges
over all finite field extensions $L'/L$ such that
$p$ does not divide the degree $[L':L]$.
The essential dimension $\ed_p(\mathcal{F})$ is then defined as
the
supremum
of
$\ed_p(\alpha)$
as $K$ ranges over
all field extensions of $k$ and
$\alpha$
ranges over $\mathcal{F}(K)$.

For generalities on essential dimension,
see~\cite{berhuy-favi, merkurjev-ed-2013, reichstein-icm, BRV2}.

\subsection{Canonical dimension}
\label{sect.cd}
An interesting example of a covariant functor $\Fields_k\to\Sets $
is the ``detection functor" $\mathcal{D}_X$ associated to
an algebraic $k$-variety $X$.
For a field extension $K/k$, we define
\[ \mathcal{D}_X(K) := \begin{cases} \text{a one-element set,
if $X$ has a $K$-point, and} \\
\text{$\emptyset$, otherwise.}
\end{cases} \]
If $k \subset K \stackrel{i}{\hookrightarrow} L$
then $0 \leqslant |\mathcal{D}_X(K)| \leqslant
|\mathcal{D}_X(L)| \leqslant 1$. Thus there is a unique morphism of sets
$ \mathcal{D}_X(K) \to \mathcal{D}_X(L)$,
which we define to be $\mathcal{D}_X(i)$.

The essential dimension (respectively, the essential $p$-dimension)
of the functor $\mathcal{D}_X$ is called the {\em canonical dimension}
of $X$ (respectively, {\em the canonical $p$-dimension} of $X$) and
is denoted by $\cd(X)$ (respectively, $\cd_p(X)$).
If $X$ is smooth and projective, then $\cd(X)$ (respectively, $\cd_p(X)$)
equals the minimal dimension of the image of a rational self-map
$X \dasharrow X$
(respectively, of a correspondence $X \rightsquigarrow X$
of degree prime to $p$).
In particular,
\begin{equation} \label{e.cd<dim}
0 \leqslant \cd_p(X) \leqslant \cd(X) \leqslant \dim(X)
\end{equation}
for any prime $p$.  If $\cd(X) = \dim(X)$, we say that $X$ is
incompressible. If $\cd_p(X) = \dim(X)$, we say that $X$
is $p$-incompressible.  For details on the notion of canonical
dimension for algebraic varieties, we refer
the reader to~\cite[\S4]{merkurjev-ed-2013}.

We will say that smooth projective varieties $X$ and $Y$ defined over $K$
are {\em equivalent} if there exist rational maps $X \dasharrow Y$ and
$Y \dasharrow X$. Similarly, we will say that $X$ and $Y$ are $p$-equivalent
for a prime integer $p$, if there exist correspondences
$X \rightsquigarrow Y$ and $Y \rightsquigarrow X$ of degree prime to $p$.

\begin{lem} \label{lem.Nishimura}
(a) If $X$ and $Y$ are equivalent, then $\cd(X) = \cd(Y)$ and $\cd_p(X) = \cd_p(Y)$
for any $p$.

\smallskip
(b) If $X$ and $Y$ are $p$-equivalent for some $p$, then $\cd_p(X) = \cd_p(Y)$.
\end{lem}

\begin{proof} (a) Let $K/k$ be a field extension.
By Nishimura's lemma, $X$ has a $K$-point if and only if so does $Y$;
see~\cite[Proposition A.6]{RY}. Thus the detection functors
$\mathcal{D}_X$ and $\mathcal{D}_Y$ are isomorphic, and
$\cd(X) = \ed(\mathcal{D}_X) = \ed(\mathcal{D}_Y) = \cd(Y)$.
Similarly, $\cd_p(X) = \cd_p(Y)$.

For a proof of part (b) see \cite[Lemma 3.6 and Remark 3.7]{snv}.
\end{proof}

\section
{Balanced algebras}

Let $Z/k$ be a Galois field extension, and $A$ be a central simple
algebra over $Z$. Given $\alpha \in \Gal(Z/k)$, we will denote the
``conjugate" $Z$-algebra $A \otimes_Z Z$, where the tensor product
is taken via $\alpha \colon Z \to Z$, by $^\alpha\! A$. We will say
that $A$ is {\em balanced over $k$} if
$^\alpha\! A$ is Brauer-equivalent to a tensor power of $A$ for
every $\alpha \in \Gal(Z/k)$.

Note that $A$ is balanced, if the Brauer class of $A$
descends to $k$\,: $^\alpha\! A$ is then isomorphic to $A$ for any $\alpha$.
In this section we will consider another family of balanced algebras.

Let $K/k$ be a field extension, $\rho \colon G \to \GL_n(K)$ be
an irreducible representation whose character $\chi$ is $k$-valued.
Recall from Theorem~\ref{thm.schur} that $\Env_k(\rho)$ is a central
simple algebra over $Z \simeq k(\chi_1) = \dots = k(\chi_n)$.

\begin{prop} \label{prop.conjugates}
$\Env_k(\rho)$ is balanced over $k$.
\end{prop}

\begin{proof}
Recall from~\cite[p.~14]{yamada2} that a cyclotomic algebra $B/Z$ is
a central simple algebra of the form
\[  B = \bigoplus_{g \in \Gal(Z(\zeta)/Z)}
Z(\zeta) u_g \, , \]
where $\zeta$ is a root of unity, $Z(\zeta)$ is a maximal subfield of $B$,
and the basis elements $u_g$ are subject to the relations
\[ \text{$u_g x = g(x) u_g$ and $u_g u_h = \beta(g, h) u_{gh} \quad$
for every $x \in Z(\zeta)$ and $g, h \in \Gal(Z(\zeta)/Z)$}. \]
Here $\beta \colon G \times G \to Z(\zeta)^*$ is a 2-cocycle whose
values are powers of $\zeta$.  Following the notational conventions
in~\cite{yamada2}, we will write $B := (\beta, Z(\zeta)/Z)$.

By \cite[Corollary 3.11]{yamada2},
$\Env_k(\rho)$ is Brauer-equivalent to some cyclotomic algebra $B/Z$,
as above.  Thus it suffices to show that every cyclotomic algebra
is balanced over $k$, i.e., $^\alpha \! B$ is Brauer-equivalent
to a power of $B$ over $Z$ for every $\alpha \in \Gal(Z/k)$.

By Theorem~\ref{thm.schur}(d), $Z$ is $k$-isomorphic to $k(\chi_1)$,
which is, by definition a subfield of $k(\zeta_e)$, where $e$ is
the exponent of $G$. Thus there is a root of unity $\epsilon$ such that
\[ Z(\zeta) \subset k(\zeta, \zeta_e) = k(\epsilon) \]
and both $\zeta$ and $\zeta_e$ are powers of $\epsilon$.
Note that $k(\epsilon)/k$ is an abelian extension, and
the sequence of Galois groups
\[ 1 \to \Gal(k(\epsilon)/Z) \to \Gal(k(\epsilon)/k) \to
\Gal(Z/k) \to 1 \]
is exact. In particular, every $\alpha \in \Gal(Z/k)$ can be lifted to
an element of $\Gal(k(\epsilon)/k)$, which we will continue to denote
by $\alpha$. Then $\alpha(\epsilon) = \epsilon^t$ for some integer $t$.
Since $\zeta$ is a power of $\epsilon$, and each $\beta(g, h)$
is a power of $\zeta$, we have
\begin{equation} \label{e.alpha}
\text{$\alpha(\beta(g, h)) = \beta(g, h)^t$ for
every $g, h \in \Gal(Z(\zeta)/k)$.}
\end{equation}
We claim that $^\alpha \! B$ is Brauer-equivalent
to $B^{\otimes t}$ over $Z$.  Indeed, since
$$
B =(\beta, Z(\zeta)/Z),
$$
we have $^\alpha \! B = (\alpha(\beta), Z(\zeta)/Z)$.
By~\eqref{e.alpha},
$^\alpha \! B = (\alpha(\beta), Z(\zeta)/Z) =
(\beta^t , Z(\zeta)/Z)$, and
$(\beta^t , Z(\zeta)/Z)$
is Brauer-equivalent to $B^{\otimes t}$, as desired.
\end{proof}

\section{Generalized Severi-Brauer varieties and
Weil transfers}

Suppose $Z/k$ is a finite Galois field extension and $A$ is
a central simple algebra over $Z$. For $1 \leqslant m \leqslant \deg(A)$,
we will denote by $\SB(A, m)$ the generalized Severi-Brauer
variety (or equivalently, the twisted Grassmannian)
of $(m-1)$-dimensional subspaces in $\SB(A)$.
The Weil transfer $\WR_{Z/k} (\SB(A, m))$
is a smooth projective absolutely irreducible
$k$-variety of dimension $[Z: k] \cdot m \cdot (\deg(A) - m)$.
For generalities on the Weil transfer, see, e.g.,~\cite{karpenko00}.

\begin{prop} \label{prop.generic} Let $Z$, $k$ and $A$ be as above,
$X := \WR_{Z/k} (\SB(A, m))$ for some
$1 \leqslant m \leqslant \deg(A)$, and
$K/k$ be a field extension.

\smallskip
(a) Write $K_Z := K \otimes_k Z$ as a direct product
$K_1 \times \dots \times K_s$,
where $K_1/Z, \dots, K_s/Z$ are field extensions. Then
$X$ has a $K$-point if and only if
the index of the central simple algebra $A_{K_i} := A \otimes_Z K_{i}$
divides $m$ for every $i = 1, \dots, s$.

\smallskip
(b) Assume that $m$ divides $\ind(A)$,
$A$ is balanced and $K = k(X)$ is the function field
of $X$. Then $K_Z = K \otimes_k Z$ is
a field, and $A \otimes_k K \simeq A \otimes _Z K_Z$ is
a central simple algebra over
$K_Z$ of index $m$.
\end{prop}

\begin{proof} First note that $A \otimes_k K \simeq A \otimes _Z K_Z$.

\smallskip
(a) By the definition of  the Weil transfer,
$X = \WR_{Z/k} (\SB(A, m))$ has a $K$-point if and only if
$\SB(A, m)$ has a $K_Z$-point or equivalently, if and only if
$\SB(A, m)$ has a $K_i$-point for every $i = 1, \dots, s$.
On the other hand, by~\cite[Proposition 3]{blanchet},
$\SB(A, m)$ has a $K_i$-point if and only if
the index of $A_{K_i}$ divides $m$.

\smallskip
(b) Since $X$ is absolutely irreducible (see, e.g.  \cite[Lemma 3]{blanchet}),
$K_Z$ is $Z$-isomorphic to
the function field of the $Z$-variety
\[ X_Z := X \times_{\Spec(k)} \Spec(Z) =
\prod_{\alpha \in \Gal(Z/k)} \SB(^\alpha \! A, m) \, , \]
see~\cite[\S2.8]{borel-serre}.  Set $F := Z(\SB(A, m))$.
By~\cite[Corollary 1]{wadsworth},
\[ \ind(A \otimes_Z F) = m \, . \]
Since $A$ is balanced, i.e., each algebra
$^\alpha \! A$ is a power of $A$,
$\ind(^\alpha \! A \otimes_Z F)$ divides $m$ for every
$\alpha \in \Gal(Z/k)$. By~\cite[Proposition 3]{blanchet},
each $\SB(^\alpha \! A, m)_F$ is rational over $F$. Thus
the natural projection of $Z$-varieties
\[ X_Z = \prod_{\alpha \in \Gal(Z/k)} \SB(^\alpha \! A, m) \to
\SB(A, m) \]
induces a purely transcendental extension of function fields
$F \hookrightarrow K_Z$. Consequently,
\[ \ind(A \otimes _Z K_Z) = \ind(A \otimes_Z F) = m \, , \]
as claimed.
\end{proof}

\section{The essential dimension of a representation}
\label{sect.def-ed}

Let us now fix a finite group $G$ and an arbitrary
field $k$, and consider the covariant functor
\[ \Rep_{G, k} :\Fields_k\to\Sets \]
defined by
$\Rep_{G, k} (K) := \{ K$-isomorphism classes of
representations $G \to \GL_n(K) \}$
for every field $K/k$. Here $n \geqslant 1$ is allowed to vary.

The essential dimension $\ed(\rho)$ of a representation
$\rho \colon G \to \GL_n(K)$ is defined by viewing $\rho$ as an
object in $\Rep_{G, k} (K)$, as in Section~\ref{sect.ed}.
That is, $\ed(\rho)$ is the smallest transcendence degree
of an intermediate field $k \subset K_0 \subset K$
such that $\rho$ is $K$-equivalent to a representation
$\rho' \colon G \to \GL_n(K_0)$.
To illustrate this notion, we include an example, where
$\ed(\rho)$ is positive and two elementary lemmas.

\begin{example} \label{ex.H}
Let $\bbH = (-1, -1) $ be the algebra of Hamiltonian quaternions over
$k = \bbR$, i.e., the $4$-dimensional $\bbR$-algebra
given by two generators $i$, $j$, subject to relations, $i^2 = j^2 = -1$
and $ij = - ji$. The multiplicative subgroup
$G  = \{ \pm 1, \pm i, \pm j, \pm ij \}$ of $\bbH$
is the quaternion group of order $8$.
Let $K = \bbR(\SB(\bbH))$, where  $\SB(\bbH)$ denotes
the Severi-Brauer variety of $\bbH$.
The representation
$\rho \colon G \hookrightarrow \bbH \hookrightarrow \bbH \otimes_{\bbR} K
\simeq \Mat_2(K)$
is easily seen to be absolutely irreducible.
We claim that $\ed(\rho) = 1$.
Indeed, $\trdeg_{\bbR}(F) = 1$,
for any intermediate extension $\bbR \subset F \subset K$,
unless $F = \bbR$. On the other hand, $\rho$ cannot descend to $\bbR$,
because $\Env_{\bbR}(\rho) = \bbH$, and thus $m_{\bbR}(\rho) =
\ind(\bbH) = 2$ by Theorem~\ref{thm.schur}(d).
\qed
\end{example}

\begin{lem} \label{lem.subadditive}
Let $G$ be a finite group, $K/k$ be a field,
$\rho_i \colon G \to \GL_{n_i}(K)$ be representations
of $G$ over $K$ (for $i = 1, \dots, s$) and
$\rho \simeq a_1 \rho_1 \oplus \dots \oplus a_s \rho_s$,
where $a_1, \dots, a_s \geqslant 1$ are integers. Then
$\ed(\rho) \leqslant \ed(\rho_1) + \dots + \ed(\rho_s)$.
\end{lem}

\begin{proof}
Suppose $\rho_i$ descends to an intermediate field
$k \subset K_i \subset K$, where $\trdeg_k(K_i) = \ed(\rho_i)$.
Let $K_0$ be the subfield of $K$ generated by $K_1, \dots, K_s$.
Then $\rho$ descends to $K_0$ and
$\ed(\rho) \leqslant \trdeg_k(K_0) \leqslant
\trdeg_k(K_1) + \dots + \trdeg_k(K_s) =
\ed(\rho_1) + \dots + \ed(\rho_s)$.
\end{proof}

\begin{lem} \label{lem.prel1}
Let $k \subset K$ be fields, $G$ be a finite group,
and $\rho \colon G \to \GL_n(K)$ be a representation. Let
$k' := k(\chi) \subset K$, where $\chi$ is the character of $\rho$.
Then the essential dimension
of $\rho$ is the same, whether we consider it as an object
on $\Rep_{K, k}$ or $\Rep_{K, k'}$.
%
\end{lem}

\begin{proof} If $\rho$ descends to an intermediate field
$k \subset F \subset K$, then $F$ automatically contains $k'$.
Moreover, $\trdeg_k(F) = \trdeg_{k'}(F)$. The rest is immediate
from the definition.
\end{proof}

\begin{lem} \label{lem.ed=0} Assume that $\Char(k)$ does not divide $|G|$
and the Schur index $m_k(\lambda)$ equals  $1$ for every absolutely
irreducible representation
$\lambda \colon G \to \GL_n(K)$, where $K$ contains $k$.
(In fact, it suffices to only consider $K = \overline{k}$.)
Then $\ed(\rho) = 0$ for any representation
$\rho \colon G \to \GL_n(L)$ over any field $L/k$.  In other words,
$\ed(\Rep_{G, k}) = 0$.
\end{lem}

\begin{proof}
Let $\chi$ be the character of $\rho$ and
$k': = k(\chi)$.  By Theorem~\ref{thm.schur}(e),
$m_{k'}(\lambda) = 1$ for every absolutely
irreducible representation $\lambda \colon G \to \GL_n(K)$
of $G$.  By Lemma~\ref{lem.prel1}
we may replace $k$ by $k' = k(\chi)$ and thus assume
that $\chi$ is $k$-valued. Corollary~\ref{cor.lift}
now tells us that $\rho$ descends to $k$.
\end{proof}

\begin{remark} \label{rem.char}
The condition of Lemma~\ref{lem.ed=0}
is always satisfied if $\Char(k) > 0$;
see~\cite[Theorem 74.9]{curtis-reiner2}.
This tells us that for non-modular
representations the notion of essential dimension
is only of interest when $\Char(k) = 0$.
The situation is drastically different in the modular setting;
see Section~\ref{sect.modular}.
\end{remark}


\section{Irreducible characters}
\label{sect.irreducible}

In view of Remark~\ref{rem.char}, we will now assume that $\Char(k) = 0$.
In this setting there is a tight connection between representations
and characters.

 \begin{lem} \label{lem.env-well-defined}
 Suppose $F_1/k$, $F_2/k$ are field extensions, and
 \[ \rho_1 \colon G \to \GL_n(F_1), \quad
 \rho_2 \colon G \to \GL_n(F_2) \]
 are representations of a finite
 group $G$, with the same character $\chi \colon G \to k$.
 Then the $k$-algebras  $\Env_k(\rho_1)$ and
 $\Env_k(\rho_2)$ are isomorphic.
 \end{lem}

 \begin{proof} Let $F/k$ be a field  containing both $F_1$ and $F_2$.
 Then $\rho_1$ and $\rho_2$ are equivalent over $F$, because
 they have the same character.  Thus $\Env_k(\rho_1)$ and
 $\Env_k(\rho_2)$ are conjugate inside $\M_n(F)$.
 \end{proof}

Given a representation $ \rho \colon G \to \GL_n(F)$,
with a $k$-valued character $\chi \colon G \to k$,
Lemma~\ref{lem.env-well-defined} tells us that,
up to isomorphism, the $k$-algebra
$\Env_k(\rho)$ depends only on $\chi$ and not on the specific
choice of $F$ and $\rho$. Thus we may denote this algebra
by $\End_k(\chi)$.

If $\rho$ is absolutely irreducible (and the character $\chi$ is
not necessarily $k$-valued), it is common to write $m_k(\chi)$ for
the index of $\End_{k(\chi)}(\chi)$ instead of $m_k(\rho)$.

 Let $\chi \colon G \to k$ be a character of $G$. Write
 \begin{equation} \label{e.char-decomposition}
 \chi = \sum_{i = 1}^r m_i \chi_i \, ,
 \end{equation}
 where $\chi_1, \dots, \chi_r \colon G \to \overline{k}$
 are absolutely irreducible and $m_1, \dots, m_r$ are non-negative integers.
 Since $\chi$ is $k$-valued, $m_i = m_j$ whenever $\chi_i$ and
 $\chi_j$ are conjugate over $k$.


\begin{lem} \label{lem.irreducible}
Let $\chi = \sum_{i = 1}^r m_i \chi_i \colon G \to k$ be a character of $G$,
as in~\eqref{e.char-decomposition}. Then the following are equivalent.

\smallskip
(a) $\chi$ is the character of a $K$-irreducible representation
$\rho \colon G \to \GL_n(K)$ for some field extension $K/k$.

\smallskip
(b) $\chi_1, \dots, \chi_r$ form a single  $\Gal(k(\chi_1)/k)$-orbit
and $m_1 = \dots = m_r$ divides $m_k(\chi_1) = \dots = m_k(\chi_r)$.
\end{lem}

\begin{proof} (a) $\Longrightarrow$ (b):
By Theorem~\ref{thm.schur}(a) and (b),
$\chi = m(\chi_1 + \dots + \chi_r)$,
where $\chi_1, \dots, \chi_r$ are absolutely irreducible characters
transitively permuted by $\Gal(K(\chi_1)/K)$, and
$m = m_K(\chi_1) = \dots = m_K(\chi_r)$.  By Lemma~\ref{lem.abelian}(b),
$\chi_1, \dots, \chi_r$ are also transitively permuted by
$\Gal(k(\chi_1)/k)$.  Moreover, by Theorem~\ref{thm.schur}(e),
$m$ divides $m_k(\chi_1) = \dots = m_k(\chi_r)$.

(b) $\Longrightarrow$ (a):
Let $K$ be the function field of the Weil transfer variety
$\WR_{Z/k}(\SB(A, m))$, where $A$ is the underlying division algebra
and $Z$ is the center of $\Env_k(\chi)$.
Since the variety $\WR_{Z/k}(\SB(A, m))$ is absolutely irreducible,
$k$ is algebraically
closed in $K$. Lemma~\ref{lem.abelian}(c) now tells us that
$\chi_1, \dots, \chi_r$
are conjugate over $K$. By Theorem~\ref{thm.schur}(c) there exists
an irreducible $K$-representation $\rho$ whose character is
$m_K(\chi_1)(\chi_1 + \dots + \chi_r)$. It remains to show that
$m_K(\chi_1) = m$. Indeed,
\[ m_K(\chi_1) = \ind (\Env_K(\chi)) = \ind (\Env_k(\chi) \otimes_k K) = m
\, . \]
Here the first equality follows from Theorem~\ref{thm.schur}(d),
the second from Lemma~\ref{lem.env}, and the third from
Proposition~\ref{prop.generic}(b).
\end{proof}

We will say that a character $\chi \colon G \to k$
is {\em irreducible over} $k$ if it satisfies the equivalent conditions
of Lemma~\ref{lem.irreducible}.

\section{The essential dimension of a character}
\label{sect.functor}

In this section we will assume that $\Char(k) = 0$
and consider subfunctors
\[  \Rep_{\chi}:\Fields_k\to\Sets \]
of $\Rep_{G, k}$
given by
\begin{multline*}
 \Rep_{\chi}(K) :=
 \\
 \text{$\{ K$-isomorphism classes of
representations $\rho \colon G \to \GL_n(K)$ with character $\chi \}$}
\end{multline*}
for every field $K/k$.  Here $\chi \colon G \to k$ is a fixed character and
$n = \chi(1_G)$.
The assumption that $\chi$ takes values in $k$ is natural in view of
Lemma~\ref{lem.prel1}, and the assumption that $\Char(k) = 0$
in view of Remark~\ref{rem.char}. Since any two $K$-representations
with the same character are equivalent, $\Rep_{\chi}(K)$ is either empty
or has exactly one element.  We will say that $\chi$ {\em can be realized
over $K/k$} if $\Rep_{\chi}(K) \neq \emptyset$. In particular,
$\Rep_{\chi}$ and $\Rep_{\chi'}$ are isomorphic if and only if
$\chi$ and $\chi'$ can be realized over the same fields $K/k$.

\begin{defn} \label{def.ed-character}
Let $\chi \colon G \to k$ be a character of
a finite group $G$ and $p$ be a prime integer.
We will refer to the essential dimension
of $\Rep_{\chi}$ as the {\em essential dimension of $\chi$}
and will denote this number by $\ed(\chi)$.
Similarly for the essential $p$-dimension:
\[ \text{$\ed(\chi) := \ed(\Rep_{\chi})$ and
$\ed_p(\chi) := \ed_p(\Rep_{\chi})$.} \]
\end{defn}

We will say that characters $\chi$ and $\lambda$ of $G$, are {\em disjoint} if
they have no common absolutely irreducible components.

\begin{lem} \label{lem.functors}
(a) If the characters $\chi, \lambda \colon G \to k$
are disjoint then $\Rep_{\chi + \lambda} \simeq
\Rep_{\chi} \times \Rep_{\lambda}$.

\smallskip
(b)
Suppose a character $\chi = \sum_{i = 1}^s m_i \chi_i
\colon G \to k$ is as in~\eqref{e.char-decomposition}. Set
$\chi' : = \sum_{i = 1}^r m_i' \chi_i$, where $m_i'$ is the greatest
common divisor of $m_i$ and $m_k(\chi_i)$. Then
$\Rep_{\chi} \simeq \Rep_{\chi'}$.
\end{lem}

\begin{proof} Let $K$ be a field extension of $k$.

\smallskip
(a) By Corollary~\ref{cor.lift},
$\chi + \lambda$ can be realized over $K$ if and only if
both $\chi$ and $\lambda$ can be realized over $K$.

\smallskip
(b)  By Corollary~\ref{cor.lift}

\smallskip
(i) $\chi$ can be realized over $K$ if and only if

\smallskip
(ii) $m_K(\chi_i)$ divides $m_i$, for every $i = 1, \dots, s$.

\smallskip
\noindent
By Theorem~\ref{thm.schur}(e),
$m_K(\chi_i)$ divides $m_k(\chi)$. Thus (ii) is equivalent to

\smallskip
(iii) $m_K(\chi_1) = \dots = m_K(\chi_s)$ divides $m'$.

\smallskip
\noindent
Applying Corollary~\ref{cor.lift} one more time, we see
that (iii) is equivalent to

\smallskip
(iv) $\chi'$ can be realized over $K$.

\smallskip
\noindent
In summary, $\chi$ can be realized over $K$ if and
only if $\chi'$ can be realized over $K$, as desired.
\end{proof}

As we observed above, $\Rep_{\chi}(K)$ has at most one element
for every field $K/k$.  In other words, $\Rep_{\chi}$ is a {\em detection
functor} in the sense of~\cite{km06} or~\cite[Section 4a]{merkurjev-ed-2013}.
We saw in Section~\ref{sect.cd} that
to every algebraic variety $X$ defined over $k$, we can
associate the detection functor $\mathcal{D}_X$, where
$\mathcal{D}_X(K)$ is either empty or has exactly one element,
depending on whether or not $X$ has a $K$-point. Given a character
$\chi \colon G \to k$, it is thus natural
to ask if there exists a smooth projective $k$-variety
$X_{\chi}$ such that the functors $\Rep_{\chi}$
and $\mathcal{D}_{X_{\chi}}$ are isomorphic. The rest of
this section will be devoted to showing that this is, indeed,
always the case. We begin by defining $X_{\chi}$.

\begin{defn} \label{def.X_chi}
(a) Let $G$ be a finite group and
$\chi:= m(\chi_1 + \dots + \chi_r) \colon G \to k$ be an irreducible
character of $G$, where $\chi_1, \dots, \chi_r$ are
$\Gal(k(\chi_1)/k)$-conjugate absolutely irreducible characters,
and $m \geqslant 1$ divides $m_k(\chi_1) = \dots = m_k(\chi_r)$.
We define the $k$-variety $X_{\chi}$ as the Weil transfer
$\WR_{Z/k}(\SB(A_{\chi}, m))$, where  $Z$ is the center and
$A_{\chi}$ is the underlying division algebra of $\Env_k(\chi)$.

\smallskip
(b) More generally, suppose $\chi := \lambda_1 + \dots + \lambda_s$, where
$\lambda_1, \dots, \lambda_s \colon G \to k$ are pairwise disjoint and
irreducible over $k$. Then we define
$X_{\chi} := X_{\lambda_1} \times_k \dots \times_k X_{\lambda_r}$,
where each $X_{\lambda_i}$ is a Weil transfer of
a generalized Severi-Brauer variety, as in part (a).
\end{defn}

\begin{thm} \label{thm.main}
Let $G$ be a finite group and $\chi := \lambda_1 + \dots + \lambda_s$
be a character, where
$$
\lambda_1, \dots, \lambda_s \colon G \to k
$$
are pairwise disjoint and irreducible over $k$. Let
$X_{\chi}$ be the $k$-variety, as in Definition~\ref{def.X_chi}.
Then the functors
$\Rep_{\chi}$ and $\mathcal{D}_{X_{\chi}}$ are isomorphic.
Consequently $\ed(\chi) = \cd(X_{\chi})$ and
$\ed_p(\chi) = \cd_p(X_{\chi})$ for any prime $p$.
\end{thm}

\begin{proof} In view of Lemma~\ref{lem.functors}(a) we may assume that
$\chi$ is irreducible over $k$, i.e., $s = 1$ and $\chi = \lambda_1$.
Write $\chi := m(\chi_1 + \dots + \chi_r)$, where
$\chi_1, \dots, \chi_r \colon G \to \overline{k}$ are
the absolutely irreducible components of $\chi$.
Let $K/k$ be a field extension.  By Corollary~\ref{cor.lift}
the following conditions are equivalent:

\smallskip
(i) $\Rep_{\chi}(K) \neq \emptyset$, i.e., $\chi$ can be realized
over $K$,

\smallskip
(ii) $m_K(\chi_j)$ divides $m$ for $j = 1, \dots, r$.

\smallskip
\noindent
Note that while the characters $\chi_1, \dots, \chi_r$ are conjugate over $k$,
they may not be conjugate over $K$.  Denote the
orbits of $\Gal(\overline{K}/K)$-action on
$\chi_1, \dots, \chi_r$ by
$\mathcal{O}_1, \dots, \mathcal{O}_t$,
and set $\mu_i := \sum_{\chi_j \in \mathcal{O}_i} \chi_j$,
so that $\chi = m(\mu_1 + \dots + \mu_t)$.

Denote the center of the central simple algebra
$\Env_k(\chi)$ by $Z$.
Write $K_Z := K \otimes_k Z$ as a direct product
$K_1 \times \dots \times K_s$,
where $K_1/Z, \dots, K_s/Z$ are field extensions,
as in Proposition~\ref{prop.generic}.
By Lemma~\ref{lem.env},
\begin{equation} \label{e.EnvK1}
\Env_{K}(\chi) \simeq
\Env_k(\chi) \otimes_Z K \simeq
\Env_k(\chi) \otimes_Z K_Z  \simeq
(\Env_{K}(\chi) \otimes_Z K_1) \times \dots \times
(\Env_{K}(\chi) \otimes_Z K_s) \, ,
\end{equation}
where $\simeq$ denotes isomorphism of $K$-algebras.
On the other hand, since $\mu_1, \dots, \mu_t$ are $K$-valued characters,
\begin{equation} \label{e.EnvK2}
\Env_{K}(\chi) \simeq \Env_K(m \mu_1) \times \dots \times \Env_K(m \mu_t) \, .
\end{equation}
Suppose $\chi_j \in \mathcal{O}_{i}$. Then by Lemma~\ref{lem.multiple}
$\Env_K(m \mu_i) \simeq \Env_K(\mu_i) \simeq
\Env_K(m_K(\chi_j) \mu_i)$,
and by Theorem~\ref{thm.schur}(d),
$\Env_K(m_K(\chi_j) \mu_i)$ is a central simple algebra of
index $m_K(\chi_j)$.  Comparing~\eqref{e.EnvK1} and~\eqref{e.EnvK2},
we conclude that
$s = t$, and after renumbering $K_1, \dots, K_s$, we may assume that
$K_i$ is the center of $\Env_K(m \mu_i)$.
Thus (ii) is equivalent to

\smallskip
(iii) the index of $\Env_K(m \mu_i) \simeq \Env_K(\chi) \otimes_Z K_i$
divides $m$ for every $i = 1, \dots, s$.

\smallskip
\noindent
By Proposition~\ref{prop.generic}(a), (iii) is equivalent to

\smallskip
(iv) $X_{\chi}$ has a $K$-point, i.e.,
$\mathcal{D}_{X_{\chi}}(K) \neq \emptyset$.

\smallskip
\noindent
The equivalence of (i) and (iv) shows that the functors
$\Rep_{\chi}$ and $\mathcal{D}_{X_{\chi}}$ are isomorphic.
Now
\[ \ed(\chi) \stackrel{\text{def}}{=}
\ed(\Rep_{\chi}) = \ed(\mathcal{D}_{X_{\chi}})
\stackrel{\text{def}}{=} \cd(X_{\chi}) \]
and similarly for the essential dimension at $p$.
\end{proof}

\begin{remark} Theorem~\ref{thm.main} can, in fact, be applied to
an arbitrary $k$-valued character $\chi \colon G \to k$. Indeed,
the character $\chi'$ of Lemma~\ref{lem.functors}(b) is a sum of
pairwise disjoint $k$-irreducible characters. Thus
$\Rep_{\chi} \simeq \Rep_{\chi'}$ by Lemma~\ref{lem.functors},
and $\Rep_{\chi'} \simeq \mathcal{D}_{X_{\chi'}}$
by Theorem~\ref{thm.main}.
\end{remark}

\section{Upper bounds}
\label{sect.upper-bounds}

In this section we will, once again, assume that $\Char(k) = 0$.
Combining Theorem~\ref{thm.main} with the inequality
\eqref{e.cd<dim}, we obtain the following

\begin{cor} \label{cor.upper}
Let $G$ be a finite group and
$\chi = m(\chi_1 + \dots + \chi_r) \colon G \to k$
be an irreducible character over $k$, as
in Section~\ref{sect.irreducible}.
Then $\ed(\chi) \leqslant \dim (X_{\chi}) = rm(m_k(\chi_1) - m)$.
\qed
\end{cor}

We are now in a position to prove upper bounds on
the essential dimension of an arbitrary representation
of a finite group $G$.

\begin{prop} \label{prop.crude-upper-bound}
Let $\rho \colon G \to \GL_n(K)$ be a
representation of a finite group $G$ over a field $K/k$.

\smallskip
(a) If $\rho$ is $K$-irreducible (but not necessarily
absolutely irreducible), then
$\ed(\rho) \leqslant n^2/4$.

\smallskip
(b) If $\rho$ is arbitrary, then
$\ed(\rho) \leqslant |G|^2/4$.

\smallskip
(c) $\ed(\Rep_{G, k}) \leq |G|^2/4$
for any base field $k$.
Here $\Rep_{G, k}$ is the functor defined at the beginning
of Section~\ref{sect.def-ed}.
\end{prop}

\begin{proof} (a) By Lemma~\ref{lem.prel1}
we may assume that the character $\chi$ of $\rho$ is $k$-valued.
By Lemma~\ref{lem.irreducible}, $\chi$ is irreducible over $k$. Write
$\chi = m(\chi_1 + \dots + \chi_r)$, where $m \ge 1$ divides
$m_k(\chi_1) = \dots = m_k(\chi_r)$.
By Corollary~\ref{cor.upper}
\begin{equation} \label{e.upper1}
\ed(\rho) \leqslant r m (m_k(\chi_1) - m) \leqslant r
\frac{m_k(\chi_1)^2}{4} \, .
\end{equation}
Now recall that by Theorem~\ref{thm.schur}(d),
$\Env_k(\rho)$ is a central simple algebra of index $m_k(\chi_1)$
over a field $Z$ such that $[Z:k] = r$.
Thus
\begin{equation} \label{e.upper2}
r m_k(\chi_1)^2 \leqslant r \dim_Z (\Env_k(\rho)) =
 \dim_k (\Env_k(\rho)) =
\dim_K (\Env_K(\rho)) \leqslant n^2 \, .
\end{equation}
Here the equality $\dim_k (\Env_k(\rho)) =
\dim_K (\Env_K(\rho))$ follows from Lemma~\ref{lem.env},
and the inequality
$\dim_K (\Env_K(\rho)) \leqslant n^2$ follows from the fact
that
$\Env_K(\rho)$ is a $K$-subalgebra of $\Mat_n(K)$.
Combining \eqref{e.upper1} and \eqref{e.upper2},
we obtain
$\ed(\rho) \leqslant  n^2/4$.

\smallskip
(b) Decompose $\rho = a_1 \rho_1 \oplus \dots \oplus a_s \rho_s$
as a direct sum of distinct $K$-irreducibles. By Lemma~\ref{lem.subadditive},
$\ed(\rho) \leqslant \ed(\rho_1) + \dots + \ed(\rho_s)$. Thus in
view of part (b) we only need to show that
\begin{equation} \label{e.upper-bound}
 \sum_{i= 1}^s \dim(\rho_i)^2  \le |G|^2 \, .
\end{equation}
Over $\overline{K}$,
$\rho_i \simeq m_i (\rho_{i1} \oplus \dots \oplus \rho_{i r_i})$,
where the $\rho_{ij}$ are distinct $\overline{K}$-irreducibles, and
$m_i$ is the common Schur index $m_K(\rho_{i1}) = \dots = m_K(\rho_{i r_i})$,
as in Theorem~\ref{thm.schur}(a). The representations
$\rho_{i1}, \dots, \rho_{i r_i}$
are conjugate over $K$ and thus have the same dimension.
By Theorem~\ref{thm.schur}(c),
the $\overline{K}$-irreducible representations
$\rho_{ij}$ are pairwise non-isomorphic, and by
Theorem~\ref{thm.schur}(f),
$m_i \leqslant \dim(\rho_{i1}) = \dots = \dim(\rho_{i r_i})$.
Thus
\[ \dim(\rho_i) = m_i r_i \dim(\rho_{i1}) \leqslant  r_i \dim(\rho_{i1})^2 \]
and
\[ \sum_{i= 1}^s \dim(\rho_i)^2  \leqslant
\sum_{i = 1}^s  r_i^2 \dim(\rho_{i1})^4 \leqslant
(\sum_{i = 1}^s r_i \dim(\rho_{i1})^2 ) ^2  =
(\sum_{i = 1}^s \sum_{j = 1}^{r_i}  \dim(\rho_{ij})^2)^2 \leqslant |G|^2. \]
Here the last inequality follows from the fact that the sum of
the squares of the dimensions of $\overline{K}$-irreducible
representations is $|G|$; see,
e.g.,~\cite[Corollary 2.4.2(a)]{serre-representations}.

(c) is just a restatement of (b); see Section~\ref{sect.ed}.
\end{proof}

\section{A variant of a theorem of Brauer}
\label{sect.brauer-ed}

A theorem of R.~Brauer~\cite{brauer} asserts for every integer
$l \geqslant 1$ there exists a number field $k$, a finite group $G$ and
an absolutely irreducible character $\chi_1$ such that the Schur index
$m_k(\chi_1) = l$. For alternative proofs of Brauer's theorem,
see~\cite{berman} or~\cite{yamada}.

In this section we will prove an analogous statement with
the Schur index replaced by the essential dimension. Note
however, that the analogy is not perfect, because the representation
we construct will not be irreducible for any $l \geqslant 2$.

\begin{prop} \label{prop.brauer-ed}
For every integer $l \geqslant 0$ there exists a number field $k$, a finite group
$G$, and a character $\chi \colon G \to k$ such that $\ed(\chi) = l$.
\end{prop}

\begin{proof}
The proposition is obvious for $l = 0$. Indeed, the trivial
representation $\rho \colon G \to \GL_1(k)$ has essential dimension $0$
for any group $G$ and any number field $k$.
We may thus assume that $l \geqslant 1$.

The strategy of the proof is as follows. We will
construct finite groups $G_1, \dots, G_l$ and
$2$-dimensional absolutely irreducible characters
$\chi_i \colon G_i \to k$, for a suitable number field $k$,
such that the Brauer classes of the quaternion algebras
$A_i := \Env_k(\chi_i)$ are
linearly independent over $\bbZ/ 2 \bbZ$ in $\Br(k)$.
(Proving linear independence for these Brauer classes
will be the most delicate part of the argument.
We will defer it to Lemma~\ref{lem.quaternion}.) We will view
each $\chi_i$ as a character of $G = G_1 \times \dots \times G_l$
via the natural projection $G \to G_i$ and
set $\chi := \chi_1 + \dots + \chi_r
\colon G \to k$. By Theorem~\ref{thm.main}
\[ \ed(\chi) = \cd (X_{\chi}) \, , \]
where $X_{\chi} : = X_{\chi_1}\times_k \dots \times_k X_{\chi_l}$,
and $X_{\chi_i}$ is the $1$-dimensional Severi-Brauer variety
$\SB(A_i)$. Since the Brauer classes of $A_1, \dots, A_l$ are
linearly independent over $\bbZ/ 2 \bbZ$ in $\Br(k)$,
\cite[Theorem 2.1]{km08} tells us that $\cd(X_{\chi}) = l$, as desired.

We now proceed with the construction of $k$, $G_1, \dots, G_l$ and
$\chi_1, \dots, \chi_l$.
Choose $l$ distinct prime integers $p_1, \dots, p_l \equiv 3 \pmod{4}$
and let
$F = \bbQ (\zeta_{p_1}, \dots, \zeta_{p_l})$, where as usual,
$\zeta_p$ denotes a primitive $p$th root of unity.  The extension
$F/\bbQ$ is Galois
with
\[ \Gal(F/\bbQ) =
\Gamma_1 \times \dots \times \Gamma_l \, , \]
where $\Gamma_i \simeq \bbZ / (p-1) \bbZ$.
Since $p_i \equiv 3 \pmod{4}$, $\Gamma_i$ has a unique
subgroup of order $2$, which we will denote by $\Gamma_i[2]$.
The non-trivial element of $\Gamma_i[2]$ takes $\zeta_{p_i}$ to $\zeta_{p_i}^{-1}$
and preserves $\zeta_{p_j}$ for every $j \neq i$.
The elementary $2$-group $\Gamma_1[2] \times \dots \times \Gamma_l[2]$
is the Sylow $2$-subgroup of $\Gamma$; we will denote it by $\Gamma[2]$.
We now set
\[ k := F^{\Gamma[2]} =
\bbQ(\zeta_{p_1} + \zeta_{p_1}^{-1}, \dots, \zeta_{p_l} + \zeta_{p_l}^{-1})
\, . \]
Let $A_i$ be the quaternion algebra
$((\zeta_{p_i} - \zeta_{p_i}^{-1})^2, -1)$ over $k$.
That is, $A_i$ is the $4$-dimensional $k$-algebra generated by
two elements, $x$ and $y$, subject to the relations
\[ \text{$x^2 = (\zeta_{p_i} - \zeta_{p_i}^{-1})^2$, $y^2 = -1$
and $x y = - yx$.} \]
Note that $k(\zeta_{p_i} - \zeta_{p_i}^{-1}) = k(\zeta_{p_i})$
is a maximal subfield of $A_i$, and
\[ G_i: = <\zeta_{p_i}> \rtimes
<y> \simeq \bbZ/ p_i \bbZ \rtimes \bbZ/ 4 \bbZ \]
is a multiplicative subgroup of $A_i^*$.
Since $k(\zeta_{p_i})$ splits $A_i$ and $k(\zeta_{p_i}) \subset F$, the inclusion
$G_i \hookrightarrow A_i$ gives rise to a $2$-dimensional
representation
\[ \rho_i \colon G_i \hookrightarrow A_i \hookrightarrow \Mat_2(F) \, .  \]
By our construction $A_i = \Env_k(\rho_i)$ and the character $\chi_i$
of $\rho_i$ is $k$-valued.  Since $\rho_i$ is faithful and
$G_i$ is a non-abelian group, this representation is absolutely
irreducible. (Otherwise $\rho_i$ would embed $G$ as a subgroup of
the abelian group $\GL_1(\overline{\bbQ}) \times \GL_1(\overline{\bbQ})$,
a contradiction.) Proposition~\ref{prop.brauer-ed} is now a
consequence of Lemma~\ref{lem.quaternion} below.
\end{proof}

\begin{remark} \label{rem.brauer-ed}
Proposition~\ref{prop.brauer-ed} implies that
there exists a field extension $K/k$ and a linear representation
$\rho \colon G \to \GL_n(K)$ such that $\ed(\rho) = l$.
Note however, that $\rho$ is not the same as
$\rho_1 \times \dots  \times \rho_l \colon G \to \GL_n(F)$,
even though they have the same character.  Indeed,
since $F/k$ is a finite extension,
$\ed(\rho_1 \times \dots  \times \rho_l) = 0$.
Under the isomorphism of functors
$\Rep_{\chi} \simeq \mathcal{D}_{X_{\chi}}$ from
Theorem~\ref{thm.main}, $\rho_1 \times \dots  \times \rho_l$
corresponds to an $F$-point of $X_{\chi}$, while
$\rho$ corresponds to the generic point.
\end{remark}

\begin{lem} \label{lem.quaternion}
Let $p_1, \dots, p_l$ be distinct prime integers,
each $\equiv 3 \pmod{4}$, and $A_i$ be the quaternion algebra
$((\zeta_{p_i} - \zeta_{p_i}^{-1})^2, -1)$
over
$k = \bbQ(\zeta_{p_1} + \zeta_{p_1}^{-1}, \dots,
\zeta_{p_l} + \zeta_{p_l}^{-1})$.
Then the classes of $A_1, \dots, A_l$
are linearly independent over $\bbZ/ 2 \bbZ$ in $\Br(k)$.
\end{lem}

\begin{proof}[Proof of Lemma~\ref{lem.quaternion}]
After renumbering $p_1, \dots, p_l$, it suffices to show that
$A_1 \otimes_k \dots \otimes_k A_s$ is not split for any
$s = 1, \dots, l$.
Since $[(a, c)] \otimes [(b, c)] = [(ab, c)]$ in $\Br(k)$,
we see that
\[ [A_1 \otimes_k \dots \otimes_k A_s] = [(T, -1)] \, . \]
Here $[A]$ denote the Brauer class of a central simple algebra $A/k$,
and
\[T := \prod_{i = 1}^s (\zeta_{p_i} - \zeta_{p_i}^{-1})^2 \in k \, . \]
Now recall that the quaternion algebra $(T, -1)$ splits over $k$ if
and only if $T$ is a norm in $k(\sqrt{-1})/k$; see,
e.g.,~\cite[Theorem 2.7]{lam}.
Thus it suffices to show that $T$ is not a norm in  $k(\sqrt{-1})/k$.
Assume the contrary: $T = N_{k(\sqrt{-1})/k}(x)$ for some $x \in k(\sqrt{-1})$.
Taking the norm in $k/\bbQ$ on both sides, we see that
\[ N_{k/\bbQ}(T) = N_{k/\bbQ}(N_{k(\sqrt{-1})/k}(x)) =
N_{k(\sqrt{-1})/\bbQ}(x) =
N_{\bbQ(\sqrt{-1})/\bbQ}(N_{k(\sqrt{-1})/\bbQ(\sqrt{-1})}(x)) \]
is a norm in $\bbQ(\sqrt{-1})/\bbQ$. Thus it suffices to prove
the following

\smallskip
Claim: $N_{k/\bbQ}(T)$ is not a norm in $\bbQ(\sqrt{-1})/\bbQ$,
i.e., is not the sum of two rational squares.

\smallskip
Our proof of this claim will be facilitated by the following diagram.
\[ \xymatrix{ &  & F \ar@{-}[d]^{\; 2^l} \ar@{-}[ddl] & = &                                       \bbQ(\zeta_{p_1}, \dots, \zeta_{p_l}) \\
         &  & k \ar@{-}[dd]^{\; \text{odd}} & =  &
\bbQ(\zeta_{p_1} + \zeta_{p_1}^{-1}, \dots,
                            \zeta_{p_l} + \zeta_{p_l}^{-1}) \\
         &  \bbQ(\zeta_{p_i}) \ar@{-}[dr]^{2} &  &  & \\
(\zeta_{p_i} - \zeta_{p_i}^{-1})^2 & \in  & k_i
\ar@{-}[d]^{\; \frac{p_i-1}{2}, \; \text{odd}} & = &
 \bbQ(\zeta_{p_i} + \zeta_{p_i}^{-1}) \\
                        &  & \bbQ &   &
                     } \]
We now proceed to compute $N_{k/\bbQ}(T)$:
\[ N_{k/\bbQ}(T) =
\prod_{i = 1}^s N_{k/\bbQ} ((\zeta_{p_i} - \zeta_{p_i}^{-1})^2) =
\prod_{i = 1}^s N_{k_i/\bbQ} ((\zeta_{p_i} - \zeta_{p_i}^{-1})^2)^{[k:k_i]}
\, , \]
and since
$(\zeta_{p_i} - \zeta_{p_i}^{-1})^2 = - N_{\bbQ(\zeta_{p_i})/k_i}(\zeta_{p_i} - \zeta_{p_i}^{-1})$, for each $i$,
\begin{gather*}
N_{k_i/\bbQ} ((\zeta_{p_i} - \zeta_{p_i}^{-1})^2) =
N_{k_i/\bbQ}
(- N_{\bbQ(\zeta_{p_i})/k_i}(\zeta_{p_i} - \zeta_{p_i}^{-1})) = \\
(-1)^{[k_i : \bbQ]}
N_{\bbQ(\zeta_{p_i})/\bbQ}(\zeta_{p_i} - \zeta_{p_i}^{-1}) =
- p_i.
\end{gather*}
The last equality follows from the identities
$\prod_{j=1}^{p_i-1} \zeta_{p_i}^j = 1$ and
$\prod_{j = 1}^{p_i-1} (1 - \zeta_{p_i}^j) = p_i$, where the latter
is obtained by substituting $t = 1$ into
$\prod_{j = 1}^{p_i-1} (t - \zeta_{p_i}^j) = t^{p_i-1} + \dots + t + 1$.
In summary,
\[ N_{k_i/\bbQ}(T) = (-1)^s \prod_{i = 1}^s p_i^{[k:k_i]} \in \bbZ. \]
If $s$ is odd then
$N_{k_i/\bbQ}(T) < 0$, and the claim is obvious. In the case where
$s$ is even, recall that by a classical theorem of Fermat,
a positive integer $n$
can be written as a sum of two rational squares if and only if
it can be written as a sum of two integer squares if and only if
every prime $p$ which is $\equiv 3 \pmod{4}$ occurs to an even power in
the prime decomposition of $n$.
Since each $p_i$ is $\equiv 3 \pmod{4}$ and each $[k_i:k]$ is odd,
we conclude that $N_{k_i/\bbQ}(T)$ cannot be written
as a sum of two rational squares for any $s = 1, \dots, l$.
This completes the proof of the claim and thus
of Lemma~\ref{lem.quaternion} and Proposition~\ref{prop.brauer-ed}.
\end{proof}

\section
{Computation of canonical $p$-dimension}
\label{sect.incompressible}

This section aims to determine canonical $p$-dimension of a broad
class of Weil transfers of generalized Severi-Brauer varieties.
Here $p$ is a fixed prime integer.  The base field $k$ is allowed
to be of arbitrary characteristic.

Let $Z/k$ be a finite Galois field extension (not necessarily abelian).
We will work with Chow motives with coefficients in a finite
field of $p$ elements; see~\cite[\S64]{EKM}.
For a motive $M$ over $Z$, $\WR_{Z/k}M$ is the motive over $k$ given
by the Weil transfer of $M$ introduced in \cite{karpenko00}.
Although the coefficient ring is assumed to be $\Z$ in~\cite{karpenko00},
and the results obtained there over $\Z$ do not formally imply similar
results for other coefficients, the proofs go through for an arbitrary
coefficient ring.

For any finite separable field extension $K/k$ and a motive $M$
over $K$, the {\em corestriction} of $M$ is a well-defined motive
over $k$; see \cite{outer}.

\begin{lem}
\label{lemma}
Let $Z/k$ be an arbitrary finite Galois field extension
and let $M_1,\dots,M_m$ be $m \geqslant 1$ motives over $Z$.
Then the motive $\WR_{Z/k}(M_1\oplus\dots\oplus M_m)$ decomposes
in a direct sum
$$
\WR_{Z/k}(M_1\oplus\dots\oplus M_m)
\simeq\WR_{Z/k}M_1\oplus\dots\oplus\WR_{Z/k}M_m\oplus N,
$$
where $N$ is a direct sum of corestrictions to $k$ of motives
over fields $K$ with $k\subsetneq K\subset Z.
$
\end{lem}

\begin{proof}
For $m=1$ the statement is void.
For $m=2$ use the same argument as
in \cite[Proof of Lemma 2.1]{qweil}.
For $m \geqslant 3$ argue by induction.
\end{proof}

Now recall from Section~\ref{sect.cd} that a $k$-variety $X$
is called incompressible if $\cd(X) = \dim(X)$ and
$p$-incompressible if $\cd_p(X) = \dim(X)$.

\begin{thm}
\label{main}
Let $p$ be a prime number, $Z/k$ a finite Galois field extension of degree $p^r$ for some $r\geqslant 0$,
$D$ a balanced central division $Z$-algebra of degree $p^n$ for some
$n\geqslant 0$, and $X$ the generalized  Severi-Brauer variety $\SB(D,p^i)$ of $D$ for some $i=0,1,\dots,n$.
Then the $k$-variety $\WR_{Z/k}X$, given by the Weil transfer of $X$, is $p$-incompressible.
\end{thm}

Note that in the case, where $Z/k$ is a quadratic Galois extension,
$D$ is balanced if the $k$-algebra given by the norm of $D$ is
Brauer-trivial; $^\alpha\! D$ for $\alpha \ne 1$ is then opposite to $D$.
In this special case Theorem~\ref{main} was proved in \cite[Theorem 1.1]{qweil}.

\begin{proof}[Proof of Theorem \ref{main}]
In the proof we will use Chow motives
with coefficients in a finite field of $p$ elements.
Therefore the Krull-Schmidt principle holds for direct summands
of motives of projective homogeneous varieties by \cite{MR2264459}
(see also \cite{upper}).

We will prove Theorem \ref{main} by induction on $r+n$.
The base case, where $r+n=0$, is trivial.
Moreover, in the case where $r=0$ (and $n$ is arbitrary),
we have $Z=k$ and thus $\WR_{Z/k}X=X$
is $p$-incompressible  by \cite[Theorem 4.3]{upper}.
Thus we may assume that $r \geqslant 1$ from now on.

If $i=n$, then $X=\Spec Z$, $\WR_{Z/k}X=\Spec k$, and the statement
of Theorem \ref{main} is trivial.  We will thus assume that $i \leqslant n-1$
and, in particular, that $n \geqslant 1$.

Let $k'$ be the function field of the variety $\WR_{Z/k}\SB(D,p^{n-1})$.
Set $Z':=k'\otimes_kZ$. By Proposition~\ref{prop.generic}(b),
the index of the central simple $Z'$-algebra
$D_{Z'}=D\otimes_Z Z'=D\otimes_kk'$ is $p^{n-1}$. Thus
there exists a central division $Z'$-algebra $D'$ such that
the algebra of ($p\times p$)-matrices over $D'$ is isomorphic to $D_{Z'}$.
Let $X'=\SB(D',p^i)$.
By \cite[Theorem 10.9 and Corollary 10.19]{MR1758562} (see also \cite{MR2110630})
and \cite[Theorems 3.8 and 4.3]{upper}, the motive of the variety $X_{Z'}$ decomposes in a direct sum
$$
M(X_{Z'})\simeq U(X')\oplus U(X')(p^{i+n-1})\oplus U(X')(2p^{i+n-1})\oplus\dots
\oplus U(X')((p-1)p^{i+n-1})\oplus N,
$$
where $U(X')$ is the upper motive of $X'$ and
$N$ is a direct sum of shifts of upper motives of the varieties
$\SB(D',p^j)$ with $j<i$.
Therefore, by Lemma \ref{lemma} and \cite[Theorem 5.4]{karpenko00}, the motive of the variety
$(\WR_{Z/k}X)_{k'}\simeq\WR_{Z'/k'}(X_{Z'})$ decomposes in a direct sum
\begin{multline}
\label{dec1}
M(\WR_{Z/k}X)_{k'}\simeq \WR_{Z'/k'}U(X')\oplus \WR_{Z'/k'}U(X')(p^{r+i+n-1})\oplus \\ \WR_{Z'/k'}U(X')(2p^{r+i+n-1})\oplus  \dots\oplus \WR_{Z'/k'}U(X')((p-1)p^{r+i+n-1})\oplus N\oplus N',
\end{multline}
where now $N$ is a direct sum of shifts of $\WR_{Z'/k'}U(\SB(D',p^j))$ with $j<i$, and
$N'$ is a direct sum of corestrictions of motives over fields $K$ with
$k'\subsetneq K\subset Z'$.
By the induction hypothesis, the variety $\WR_{Z'/k'}X'$ is $p$-incompressible.
By \cite[Theorem 5.1]{canondim}, this means that no positive shift of the motive $U(\WR_{Z'/k'}X')$ is a direct summand of the motive of $\WR_{Z'/k'}X'$.
It follows by \cite{outer} that $\WR_{Z'/k'}U(X')$ is a direct sum of $U(\WR_{Z'/k'}X')$, of $U(\WR_{Z'/k'}\SB(D',p^j))$ with $j<i$, and
of corestrictions of motives over fields $K$ with $k'\subsetneq K\subset Z'$.
Therefore we may exchange $\WR_{Z'/k'}$ with $U$ in (\ref{dec1}) and get
a decomposition of the form
\begin{multline}
\label{dec2}
M(\WR_{Z/k}X)_{k'}\simeq U(\WR_{Z'/k'}X')\oplus U(\WR_{Z'/k'}X')(p^{r+i+n-1})\oplus \\ U(\WR_{Z'/k'}X')(2p^{r+i+n-1})\oplus  \dots\oplus U(\WR_{Z'/k'}X')((p-1)p^{r+i+n-1})\oplus N\oplus N',
\end{multline}
where $N$ is now a direct sum of shifts of
some $U(\WR_{Z'/k'}\SB(D',p^j))$ with $j<i$, and
$N'$ is a direct sum of corestrictions of motives over fields $K$ with
$k'\subsetneq K\subset Z'$.
Note that the first $p$ summands of decomposition (\ref{dec2}) (that is, all but the last two) are shifts of an indecomposable motive; moreover, no shift of this motive is isomorphic to a summand of $N$ or of $N'$.
Since the variety $\WR_{Z'/k'}X'$ is $p$-incompressible, we have
$$
\dim U(\WR_{Z'/k'}X')=\dim \WR_{Z'/k'}X'=[Z':k']
\cdot\dim X'=p^r\cdot p^i(p^{n-1}-p^i) \, .
$$
(We refer the reader to~\cite[Theorem 5.1]{canondim} for the definition
of the dimension of the upper motive, as well as its relationship
to the dimension and $p$-incompressibility of the corresponding variety).
Note that the shifting number of the $p$-th summand
in (\ref{dec2}) plus $\dim\WR_{Z'/k}X'$ equals $\dim\WR_{Z/k}X$:
$$
(p-1)p^{r+i+n-1}+p^rp^i(p^{n-1}-p^i)=p^rp^i(p^n-p^i).
$$

We want to show that the variety $\WR_{Z/k}X$ is $p$-incompressible.
In other words, we want to show that $\dim U(\WR_{Z/k}X)=\dim \WR_{Z/k}X$.
Let $l$ be the number of shifts of $U(\WR_{Z'/k'}X')$ contained in the complete decomposition of
the motive $U(\WR_{Z/k}X)_{k'}$.
Clearly, $1 \leqslant l \leqslant p$ and
it suffices to show that $l=p$ because in this case the $p$-th summand of (\ref{dec2})
is contained in the complete decomposition of $U(\WR_{Z/k}X)_{k'}$.

The complete motivic decomposition of $\WR_{Z/k}X$ contains several shifts of $U(\WR_{Z/k}X)$.
Let $N$ be any of the remaining (indecomposable) summands.
Then, by \cite{outer}, $N$ is either a shift of the upper motive $U(\WR_{Z/k}\SB(D,p^j))$ with some $j<i$ or a corestriction to $k$ of a motive over a field $K$ with $k\subsetneq K\subset Z$.
It follows that the complete decomposition of $N_{k'}$ does not contain any shift of $U(\WR_{Z'/k'}X')$.
Therefore $l$ divides $p$, that is, $l=1$ or $l=p$, and we only need to show that $l\ne1$.

We claim that $l>1$ provided that $\dim U(\WR_{Z/k}X)>\dim U(\WR_{Z'/k'}X')$.
Indeed, by \cite[Proposition 2.4]{sgog}, the complete decomposition of $U(\WR_{Z/k}X)_{k'}$ contains as a summand the motive $U(\WR_{Z'/k'}X')$ shifted by the difference $\dim U(\WR_{Z/k}X)-\dim U(\WR_{Z'/k'}X')$.
Therefore, in order to show that $l\ne1$ it is enough to show that
$$
\dim U(\WR_{Z/k}X)>\dim U(\WR_{Z'/k'}X').
$$

We already know the precise value of the dimension on the right, so we only need to find a good enough lower bound on the dimension on the left.
This will be given by $\dim U((\WR_{Z/k}X)_{\tilde{k}})$, where $\tilde{k}/k$ is a degree $p$ Galois field subextension of $Z/k$.
We can determine the latter dimension using the induction hypothesis.

Indeed, since $\WR_{Z/k}X\simeq\WR_{\tilde{k}/k}\WR_{Z/\tilde{k}}X$, the variety $(\WR_{Z/k}X)_{\tilde{k}}$ is isomorphic to
$$
(\WR_{Z/k}X)_{\tilde{k}}\simeq\Prod_{\tilde{\alpha}\in\tilde{\Gamma}}{}^{\tilde{\alpha}}\!\WR_{Z/\tilde{k}}X\simeq
\WR_{Z/\tilde{k}}\Prod_{\tilde{\alpha}\in\tilde{\Gamma}}{}^\alpha\!X,
$$
where $\Gamma$ is the Galois group of $Z/k$,
$\tilde{\Gamma}$ is the Galois group of $\tilde{k}/k$,
and $\alpha\in \Gamma$ is a representative of
$\tilde{\alpha}\in\tilde{\Gamma}$
(see~\cite[\S2.8]{borel-serre}).  Since $D$ is balanced,
the product
$\Prod_{\tilde{\alpha}\in\tilde{\Gamma}}{}^\alpha\!X$
is equivalent to $X$. Hence, by Lemma~\ref{lem.Nishimura} these
varieties have the same canonical $p$-dimension
(i.e., the dimensions of their upper motives coincide).
The latter variety is $p$-incompressible
by the induction hypothesis.  Consequently,
$$
\dim U(\WR_{Z/k}X) \geqslant \dim U((\WR_{Z/k}X)_{\tilde{k}})=\dim\WR_{Z/\tilde{k}}X=p^{r-1}\cdot p^i(p^n-p^i).
$$
The lower bound $p^{r-1}\cdot p^i(p^n-p^i)$ on $\dim U(\WR_{Z/k}X)$ thus obtained is good enough for our purposes, because
$$
p^{r-1}\cdot p^i(p^n-p^i)>p^r\cdot p^i(p^{n-1}-p^i)=\dim U(\WR_{Z'/k'}X').
$$
This completes the proof of Theorem \ref{main}.
\end{proof}

The following example, due to A. Merkurjev, shows that
Theorem~\ref{main} fails if $D$ is not assumed to be balanced.

\begin{example}
\label{example1}
Let $L$ be a field containing a primitive $4$-th root of unity.
Let $Z$ be the field $Z:=L(x,y,x',y')$ of rational functions over $L$
in four variables $x,y,x',y'$.
Consider the degree $4$ cyclic central division $Z$-algebras
$C:=(x,y)_4$ and $C':=(x',y')_4$.
Let $k\subset Z$ be the subfield $Z^\alpha$ of the elements in $Z$
fixed under the $L$-automorphism $\alpha$ of $Z$ exchanging $x$ with
$x'$ and $y$ with $y'$.
The field extension $Z/k$ is then Galois of degree $2$, and the
algebra $C'$ is conjugate to $C$.

The index of the tensor product of $Z$-algebras $C\otimes
{C'}^{\otimes2}$ is $8$.
Let $D/Z$ be the underlying (unbalanced!) division algebra of degree $8$.
The subgroup of the Brauer group $\Br(Z)$ generated by the classes of
$D$ and $^\alpha\!D=C'\otimes C^{\otimes 2}$ coincides with the
subgroup generated by the classes of $C$ and $^\alpha\!C=C'$.
Therefore the
varieties
$X_1 := \WR_{Z/k}\SB(D)$ and $X_2 := \WR_{Z/k}\SB(C)$ are
equivalent.
Thus, by Lemma \ref{lem.Nishimura},
\[ \cd(X_1) = \cd(X_2) \leqslant \dim(X_2) < \dim(X_1) \]
and consequently, $X_1$ is compressible (and in particular,
$2$-compressible).
\end{example}

\section{Some consequences of Theorem~\ref{main}}
\label{sect.formula}

Theorem \ref{main} makes it possible to determine the canonical
$p$-dimension of the Weil transfer in the situation, where
the degrees of $Z/k$ and of $D$ are not necessarily $p$-powers.

\begin{cor}
\label{cormain}
Let $Z/k$ be a finite Galois field extension and $D$ a balanced
central division $Z$-algebra. 
For any positive integer $m$ dividing $\deg(D)$, one has
$$
\cd_p\WR_{Z/k}\SB(D,m)=\dim\WR_{Z/k'}\SB(D',m')=[Z:k']\cdot m'(\deg D'-m'),
$$
where $m'$ is the $p$-primary part of $m$ (i.e., the highest power of
$p$ dividing $m$), $D'$ is the $p$-primary component of $D$, and
$k'=Z^{\Gamma_p}$, where $\Gamma_p$ is a Sylow $p$-subgroup of
$\Gamma:=\Gal(Z/k)$
(so that $[Z:k']$ is the $p$-primary part of $[Z:k]$).
\end{cor}

\begin{proof}
Since the degree $[k':k]$ is prime to $p$, we have
$$
\cd_p\WR_{Z/k}\SB(D,m)=\cd_p (\WR_{Z/k}\SB(D,m))_{k'} \, ;
$$
see \cite[Proposition 1.5(2)]{ASM-ed}.
The $k'$-variety $\WR_{Z/k}\SB(D,m)_{k'}$ is isomorphic to
a product of $\WR_{Z/k'}\SB(D,m)$ with several
varieties of the form
$\WR_{Z/k'}\SB(\tilde{D},m)$ where
$\tilde{D}$ ranges over a set of conjugates of $D$.
Since $D$ is balanced, these algebras $\tilde{D}$ are
Brauer-equivalent to powers of $D$. Thus
the product is equivalent to the
$k'$-variety $\WR_{Z/k'}\SB(D,m)$.
We conclude by Lemma \ref{lem.Nishimura}
that $\cd_p\WR_{Z/k}\SB(D,m)=\cd_p\WR_{Z/k'}\SB(D,m)$.
In the sequel we will replace $k$ by $k'$, so that the degree
$[Z:k]$ becomes a power of $p$.

We may now replace $k$ by its $p$-special closure;
see~\cite[Proposition 101.16]{EKM}. This will not change
the value of $\cd_p(X)$. In other words, we may assume that $k$
is $p$-special. Under this assumption the algebras $D$ and $D'$ become
Brauer-equivalent and consequently, the $k$-varieties $\WR_{Z/k}\SB(D,m)$ and
$\WR_{Z/k}\SB(D',m')$ become equivalent. By Lemma~\ref{lem.Nishimura},
\[ \cd_p\WR_{Z/k}\SB(D,m)=\cd_p\WR_{Z/k}\SB(D',m') . \]
Since the $Z$-algebra $D'$ is balanced over $k$,
Theorem~\ref{main} tells us that $\WR_{Z/k}\SB(D',m')$
is $p$-incompressible. That is,
\[ \cd_p\WR_{Z/k}\SB(D',m') =
\dim (\WR_{Z/k}\SB(D',m')) = [Z:k]\cdot m'(\deg D'-m') \, , \]
and the corollary follows.
\end{proof}

\begin{remark} \label{rem.main2} Corollary~\ref{cormain} can be used
to compute the $p$-canonical dimension of
$\WR_{Z/k}\SB(D,j)$ for any $j = 1, \dots, \deg(D)$, even if $j$ does not
divide $\deg(D)$. Indeed, let $m$ is the greatest common divisor of $j$
and $\deg(D)$. Proposition~\ref{prop.generic}(a) tells us that
for any field extension $K/k$,
$\WR_{Z/k}\SB(D,j)$ has a $K$-point if and only if
$\WR_{Z/k}\SB(D,m)$ has a $K$-point. In other words,
the detection functors for these two varieties are isomorphic.
Consequently,
\[ \text{$\cd(\WR_{Z/k}\SB(D,j)) = \cd(\WR_{Z/k}\SB(D,m))$ and
$\cd_p(\WR_{Z/k}\SB(D,j)) = \cd_p(\WR_{Z/k}\SB(D,m))$,} \]
and the value of
$\cd_p(\WR_{Z/k}\SB(D,m))$ is given by Corollary~\ref{cormain}.
\end{remark}

We now return to the setting of
Sections~\ref{sect.irreducible}--\ref{sect.upper-bounds}.
In particular, $G$ is a finite group, and the base field $k$ is
of characteristic $0$.

\begin{cor} \label{cor.main2}
Let $\chi = m(\chi_1 + \dots + \chi_r) \colon G \to k$
be an irreducible $k$-valued character,
where $\chi_1, \dots, \chi_r$ are absolutely irreducible and
conjugate over $k$, and $m$ divides $m_k(\chi_1) = \dots = m_k(\chi_r)$,
as in Section~\ref{sect.irreducible}.

\smallskip
(a) $\ed_p(\chi) = r' m' (m_k(\chi_1)' - m')$.
Here $x'$ denotes the
$p$-primary part of $x$ (i.e., the highest power of $p$ dividing $x$)
for any integer $x \ge 1$.

\smallskip
(b) If $r$ and $m_k(\chi_1)$ are powers of $p$, then
$\ed_p(\chi) = \ed(\chi) = \dim(X_{\chi}) = rm(m_k(\chi_1) - m)$.
Here $X_{\chi}$ is as in Definition~\ref{def.X_chi}.
\end{cor}

\begin{proof} (a) Let $D$ be the underlying division algebra and
$Z/k$ be the center of $\Env_k(\chi)$. By Theorem~\ref{thm.main},
$\ed_p(\chi) = \cd_p (X_{\chi})$. By Proposition~\ref{prop.conjugates},
$D$ is balanced. The desired conclusion now follows from
Corollary \ref{cormain}.

(b) Here $r' = r$, $m_k(\chi_1)' = m_k(\chi)$ and thus $m' = m$.  By part (a),
\[ \dim(X_{\chi}) = rm(m_k(\chi_1) - m) =
\ed_p(\chi) \leqslant \ed(\chi) \, . \]
On the other hand, by Corollary~\ref{cor.upper},
$\ed(\chi) \leqslant rm(m_k(\chi_1) - m)$, and part (b) follows.
\end{proof}

\begin{remark} While a priori
$\ed_p(\chi)$ depends on $k$, $G$, and $\chi$,
Corollary~\ref{cor.main2}(a)
shows that, in fact, $\ed_p(\chi)$ depends only
on the integers $r$, $m$, and $m_k(\chi_1)$. (Here we are
assuming that $\chi$ is irreducible.)
We do not know if the same is true of $\ed(\chi)$.
\end{remark}

\begin{remark} We do not know a formula for the
canonical $p$-dimension of a product of Weil transfers of
generalized Severi-Brauer varieties similar to Corollary~\ref{cormain},
except for~\cite[Theorem 2.1]{km08}, in the case, where $Z = k$ and $m = 1$.
(Note that this special case played a key role in the proof
of Proposition~\ref{prop.brauer-ed}.) For this reason we do not know
how to generalize Corollary~\ref{cor.main2} to the case, where
the character $\chi$ is not $k$-irreducible.
\end{remark}

\section{A variant of a theorem of Schilling}
\label{sect.schilling-ed}

Let $G$ be a $p$-group and $\chi_1$ be an absolutely irreducible
character of $G$. It is well known
that for any field $k$ of characteristic
$0$,
$m_k(\chi_1) = 1$ if $p$ is odd, and $m_k(\chi_1) =1$ or $2$ if $p = 2$.
Following C.~Curtis and I.~Reiner, we will
attribute this theorem to O.~Schilling;
see~\cite[Theorem 74.15]{curtis-reiner2}.
For further bibliographical references,
see~\cite[Corollary 9.8]{yamada2}.

In this section we will use Corollary~\ref{cor.main2}
to prove the following analogous statement,
with the Schur index replaced by the essential dimension.

\begin{prop} \label{prop.schilling-ed}
Let $k$ be a field of characteristic $0$,
$G$ be a $p$-group, and $\chi \colon G \to k$ be
an irreducible character over $k$.

\smallskip
(a) If $p$ is odd then $\ed(\chi) = 0$.

\smallskip
(b) If $p = 2$ then
$\ed_2(\chi) = \ed(\chi) = 0$ or $2^l$ for some integer $l \geqslant 0$.

\smallskip
(c) Moreover, every $l \geqslant 0$ in part (b) can occur with
$k = \bbQ$, for suitable choices of $G$ and $\chi$.
\end{prop}

\begin{proof}
Write $\chi = m(\chi_1 + \dots + \chi_r)$, where
$\chi_i \colon G \to \overline{k}$ are absolutely
irreducible characters and $m$ divides $m_k(\chi_1)$.
If $m = m_k(\chi_1)$ then $\ed(\chi) = 0$ by
Corollary~\ref{cor.upper}.

(a) In particular, this will always be the case if $p$ is odd.
Indeed, by Schilling's theorem, $m_k(\chi_1) = 1$ and thus $m = 1$.
(Also cf. Lemma~\ref{lem.ed=0}.)

\smallskip
(b) By Schilling's theorem, $m_k(\chi) = 1$ or $2$, and by
the above argument, we may assume that $m < m_k(\chi_1)$.
Thus the only case we need to consider is $m_k(\chi_1) = 2$
and $m = 1$.  By Theorem~\ref{thm.schur}(b), $r = [k(\chi_1):k]$. Since
$k(\chi_1) \subset k(\zeta_e)$, where the exponent $e$ of $G$ is
a power of $2$, we see that $r$ divides $[k(\zeta_e): k]$, which is,
once again, a power of $2$. Thus we conclude that $r$ is a power of $2$.
Corollary~\ref{cor.main2}(b) now tells us that
\begin{equation} \label{e.schilling}
\ed_2(\chi) = \ed(\chi) = r m (m_k(\chi) - m) =
r \cdot 1 \cdot (2-1) = r
\end{equation}
is a power of $2$, as claimed.

\smallskip
(c) Let $s = 2^{l+2}$, and $\sigma \in \Gal(\bbQ(\zeta_s)/\bbQ)$
be complex conjugation, and
\[ F := \bbQ(\zeta_s)^{\sigma} = \bbQ(\zeta_s) \cap \bbR =
\bbQ(\zeta_s + \zeta_s^{-1}) \, . \]
Consider the quaternion algebra $A = ((\zeta_s - \zeta_s^{-1})^2, -1)$
over $F$, i.e., the $F$-algebra generated by
elements $x$ and $y$, subject to the relations
\[ \text{$x^2 = (\zeta_{s} - \zeta_{s}^{-1})^2$, $y^2 = -1$
and $x y = - yx$.} \]
Arguing as in the proof of Proposition~\ref{prop.brauer-ed}, we see that
$F(\zeta_{s} - \zeta_{s}^{-1}) = \bbQ(\zeta_{s})$
is a maximal subfield of $A$,
$\zeta_{s}$ and $y$ generate a
a multiplicative subgroup of $A$ of order $2s$, which spans
$A$ as an $F$-vector space,
and the inclusion $G \hookrightarrow A$ gives rise to
an absolutely irreducible
$2$-dimensional representation
\[ \rho_1 \colon G \hookrightarrow A \hookrightarrow \GL_2(F)
\, .  \]
Denote the character of $\rho_1$ by $\chi_1 \colon G \to F$.
Since $F$ is generated by $\chi_1(\zeta_s) = \zeta_s + \zeta_s^{-1}$
as a field extension of $\bbQ$, we have $\bbQ(\chi_1) = F$.
Thus $\chi_1$ has exactly
\[ r = [F : \bbQ] = \dfrac{1}{2} [\bbQ(\zeta_s): \bbQ] = 2^l \]
conjugates over $\bbQ$, and $\chi = \chi_1 + \dots + \chi_r$ is an
irreducible character over $\bbQ$.

Note that since $s = 2^{l+2} \geqslant 4$,
$(\zeta_{s} - \zeta_{s}^{-1})^2 < 0$,
$A \otimes_ F \bbR$ is $\bbR$-isomorphic to the Hamiltonian
quaternion algebra $\bbH = (-1, -1)$ and hence, is non-split.
Thus $\ind(A) = 2$. Since $A = \Env_{\bbQ}(\rho)$,
Theorem~\ref{thm.schur}(d) tells us that $m_{\bbQ}(\chi_1) = 2$.
Applying Corollary~\ref{cor.main2}(b), as in~\eqref{e.schilling},
we conclude that $\ed_2(\chi) = \ed(\chi) = r = 2^l$,
as desired.
\end{proof}

\section{Essential dimension of modular representations}
\label{sect.modular}

Let $G$ be a finite group and $\Rep_{G, k}$ be the functor of
representations defined at the beginning of Section~\ref{sect.def-ed}.
In the non-modular setting (where $\Char(k)$ does not divide $|G|$),
we know that
\[ \text{$\ed(\Rep_{G, k})$ is }
\begin{cases} \text{$0$, if $\Char(k) > 0$, by Remark~\ref{rem.char}, and} \\
\text{$\leqslant
|G|^2/4$, if $\Char(k) = 0$,
by Proposition~\ref{prop.crude-upper-bound}.}
\end{cases}
\]
We shall now see that essential dimension of representations
behaves very differently in the modular case.

\begin{prop} \label{prop.modular}
Let $k$ be a field of characteristic $p$. Suppose a finite group
$G$ contains an elementary abelian subgroup $E
= \langle g_1, g_2 \rangle \simeq (\bbZ / p \bbZ)^2$ of rank $2$.
Then $\ed(\Rep_{G, k}) = \infty$.
\end{prop}

\begin{proof} Let $\Sub_{\bbP^1, k} \colon \Fields_k \to \Sets$
be a covariant functor, given by
\[ \text{$\Sub_{\bbP^1, k} (K) := \{$closed subvarieties
of $\bbP^{1}_K\}$.} \]
Here subvarieties of $\bbP^{1}_K$ are required to be reduced
but not necessarily irreducible.  Closed subvarieties
$X, Y \subset \bbP^{1}_K$ represent the same element in $\Sub_{\bbP^1, k}(K)$
if $X(\overline{K}) = Y(\overline{K})$ in $\bbP^{1}(\overline{K})$.
We will now consider the morphism of functors
\[ \V_{E, k} \colon \Rep_{G, k} \to \Sub_{, \bbP^1, k} \]
which associates to a representation $\rho \colon G \to \GL_n(K)$,
the rank variety $V_E(\rho)$, as defined by J.~Carlson~\cite{carlson}.
Recall that in projective coordinates $(x_1: x_2)$ on $\bbP^{1}_K$,
the rank variety $V_E(\rho)$ is given by
\begin{equation} \label{e.rank}
\rank(A_{x_1, x_2}) < \frac{(p-1) n } {p} \, ,
\end{equation}
where $A = x_1 (\rho(g_1) - I_n) + x_2 (\rho(g_2) - I_n)
\in \Mat_n(K[x_1, x_2])$, and $I_n$ is the $n \times n$ identity
matrix. In other words, condition~\eqref{e.rank} is equivalent to
the vanishing of certain minors of $A_{x_1, x_2}$. These
minors are homogeneous polynomials in $K[x_1, x_2]$,
and $V_E(\rho)$ is the reduced subvariety of $\bbP^{1}_K$
they cut out.  Note that the generators $g_1, g_2$
of $E$ are assumed to be fixed throughout.  For details
on this construction, see~\cite[Section 4]{carlson}
or~\cite[Section 5.8]{benson2}.

\smallskip
To make the rest of the proof more transparent, we will first consider
the case, where $G = E$. In this case it follows from the work
of Carlson that the functor $\V_{E, k}$ is surjective. That is,
for any given field $K/k$, every reduced subvariety
$X \subset \bbP^{1}_K$ can be realized as the rank variety
of a suitable representation $\rho \colon E \to \GL_n(K)$;
see~\cite[Corollary 5.9.2]{benson2}. Thus
$\ed(\Rep_{G, k}) \geqslant \ed(\Sub_{\bbP^1, k})$;
see~\cite[Lemma 1.9]{berhuy-favi}.  It now suffices
to show that  $\ed(\Sub_{\bbP^1, k}) = \infty$.  Let
$L/k$ be a field, $a_1, \dots, a_n \in L$, and
$X[n]$ be the union of the points
\begin{equation} \label{e.zero-cycle}
X_1 = (1: a_1), \dots, X_n = (1: a_n)
\end{equation}
in $\bbP^1$. We view $X[n]$ as an element of $\Sub_{\bbP^1, k}(L)$.

\smallskip
Claim: Suppose $X[n]$ descends to a subfield $K \subset L$.
Then $a_{i}$ is algebraic over $K$ for every $i = 1, \dots, n$.

\smallskip
By the definition of the functor $\Sub_{\bbP^1, k}$, $X[n]$ descends to $K$ if
$X[n]$ can be cut out (set-theoretically) by
homogeneous polynomials $f_1, \dots, f_s \in K[x_1, x_2]$.
In other words, the points $X_1, \dots, X_n$ in~\eqref{e.zero-cycle}
are the only
non-zero solutions, in the algebraic closure $\overline{L}$,
of a system of homogeneous equations
\[ f_1(x_1, x_2) = \dots = f_s(x_1, x_2) = 0 \]
with coefficients in $K$. Since every solution of such
a system can be found over $\overline{K}$,
we have $a_1, \dots, a_n \in \overline{K}$. This proves the claim.

Taking $a_1, \dots, a_n$ to be independent variables and
$L := k(a_1, \dots, a_n)$, we see
that $\trdeg_k(K) = \trdeg_k(L) = n$ and thus in this
case $\ed(X[n]) = n$.  Therefore,
\[ \ed(\Sub_{\bbP^1, k}) \geqslant \sup_{n \geqslant 1} \ed(X[n]) = \infty \, . \]
This completes the proof of the proposition in the case
where $G = E$.

\smallskip
We now proceed with the proof of Proposition~\ref{prop.modular}
for a general group $G$. Denote the centralizer and the normalizer
of $E$ in $G$ by $C_G(E)$ and $N_G(E)$, respectively. Then
$W_G(E) := N_G(E)/C_G(E)$ acts on $\bbP_k^{1}$.
By the Quillen Stratification Theorem~\cite[Theorem 5.6.3]{benson2}
and \cite[Corollary 5.9.2]{benson2}, there exists
a closed $W_G(E)$-invariant $k$-subvariety $B \subsetneq \bbP_k^{1}$
with the following property.
Every $X \in \Sub_{\bbP^1, k}(L)$ satisfying conditions

\smallskip
(i) $X$ is $W_{G}(E)$-invariant, and

\smallskip
(ii) no irreducible component of $X$ is contained in $B_K$,

\smallskip
\noindent
lies in the image of $\V_{E, k}(L)$.
(The subvariety $B \subset \bbP^{1}_k$ comes from the cohomology
of elementary abelian subgroups of $G$ strictly contained in $E$; see
the bottom of~\cite[p. 178]{benson2}.)
Consequently, for any $X \in \Sub_{\bbP^1, k}(L)$ satisfying (i) and (ii),
$\ed(\Rep_{G, k}) \geqslant \ed(X)$.

In particular, let $L := k(a_1, \dots, a_n)$, where
$a_1, \dots, a_n$ are independent variables over $k$,
$X_1, \dots, X_n \subset \bbP^1_L$ be as in~\eqref{e.zero-cycle},
and \[ Z[n] := \bigcup \, w(X_i) \, , \]
where the union is taken over all
$i = 1, \dots, n$ and all $w \in W_G(E)$. Note that every point
of the form $w(X_i)$ corresponds to a dominant $k$-morphism
$\Spec(L) \to \bbP_k^{1}$ and hence, cannot lie in $B_K$ for
any proper subvariety $B \subset \bbP_k^{1}$.  Thus $Z[n]$ satisfies
conditions (i) and (ii). We conclude that
$Z[n]$ lies in the image of $\V_{E, k}$, and thus
\begin{equation} \label{e.edZ[n]}
\ed(\Rep_{G, k}) \geqslant \ed(Z[n])
\end{equation}
for every $n \geqslant 1$.  On the other hand,
the Claim above shows that if $Z[n]$ descends to an intermediate field
$k \subset K \subset L$ then every $a_i$ is algebraic over $K$.
Hence, $\ed(Z[n]) = \trdeg_k(L) = n$, and~\eqref{e.edZ[n]}
tells us that $\ed(\Rep_{G, k}) = \infty$, as desired.
\end{proof}

\bigskip
\noindent
{\sc Acknowledgements.}
The authors are grateful to Alexander Merkurjev for contributing
Example~\ref{example1} and to Patrick Brosnan, Jon Carlson,
Jerome Lefebvre, Julia Pevtsova, and Lior Silberman for stimulating
discussions.


\def\cprime{$'$}

\end{document}